\documentclass[11pt]{siamltex}
\usepackage{amsmath}
\usepackage{graphicx}
\usepackage{bm}
\usepackage{amssymb,version}
\usepackage{cases}
\usepackage{microtype}
\usepackage{color}
\usepackage{enumerate}
\usepackage{xcolor} 
\usepackage{amssymb}
\newtheorem{rem}{Remark}[section]

\allowdisplaybreaks
\begin{document}
\title{A unified framework of fully decoupled, bound-preserving and energy-dissipative schemes for two-phase flow in porous media
\thanks{The work of X. Li was supported in part by the National Natural Science Foundation of China (Grant Nos. 12271302, 12131014) and Shandong Provincial Natural Science Foundation for Outstanding Youth Scholar (Grant No. ZR2024JQ030). The work of C. Wang was supported in part by NSF grants DMS-2309548.}}

	\author{Xiaoli Li\thanks{School of Mathematics, Shandong University, Jinan, Shandong, 250100, P.R. China. Email: xiaolimath@sdu.edu.cn}.	
		\and Cheng Wang\thanks{Mathematics Department, University of Massachusetts, North Dartmouth, MA 02747 USA. Email: cwang1@umassd.edu}.
		\and Yujing Yan\thanks{Corresponding author. School of Mathematics, Shandong University, Jinan, Shandong, 250100, P.R. China. Email: yyjsust@163.com}.
		\and Nan Zheng\thanks{Department of Applied Mathematics, The Hong Kong Polytechnic University, Hung Hom, Kowloon, Hong Kong. Email: nanzheng@polyu.edu.hk}.
	}
	\maketitle
	
	\begin{abstract}
Developing high-order numerical schemes for two-phase flow in porous media that preserve key physical properties remains a significant challenge in numerical analysis. In this article, we propose a general framework to construct fully discrete first- and second-order numerical schemes for thermodynamically consistent model of incompressible and immiscible two-phase flow in porous media. The proposed schemes are rigorously proved to ensure five fundamental properties: (i) unique solvability; (ii) full decoupling; (iii) bound preservation for both phases; (iv) original energy dissipation; (v) local mass conservation for both phases. The key to ensure the unique solvability lies in guaranteeing the strict convexity of the discrete energy functionals associated with the constructed schemes. Departing from the coupled solution approach for the pressure and saturation variables, the proposed approach breaks traditional paradigm by subtracting the two-phase mass conservation equations to derive a fully decoupled system. In addition, the bound-preserving property for both phases is established by leveraging the singular nature of the logarithmic term around the limit values of $0$ and $1$. A rigorous error estimate for the first-order scheme, in the $\ell^{\infty}(0,T; H_h^{-1} (\Omega)) \cap \ell^{2}(0,T; \ell^2(\Omega))$ norm for the saturations of two phases, is established. Finally, various numerical examples are presented to verify the theoretical results and demonstrate the efficiency of the proposed schemes.
	\end{abstract}

	
	\pagestyle{myheadings}
	\thispagestyle{plain}

\section{Introduction}
Numerical modeling of two-phase flow is of great interest in petroleum reservoir engineering \cite{chavent1986mathematical,MURAD2013192,leng2025hdg} and groundwater hydrology \cite{CHO2016184,kim1997numerical}. \textcolor{black}{A key challenge in the numerical design and simulation lies in the preservation of the intrinsic mathematical structures and physical laws of the underlying models.} Although  classical models have been extensively studied \cite{CHEN20141,chen2006computational,GALUSINSKI20081741,HOTEIT200856}, \textcolor{black}{they often exhibit a notable limitation:} the energy formulations typically account only for the ideal contributions of wetting-phase saturation. This omission of non-ideal interfacial effects constitutes a fundamental limitation in achieving thermodynamic consistency. 
It is worth mentioning that a thermodynamically consistent model for two-phase incompressible and immiscible flow in porous media has been developed in \cite{gao2020thermodynamically}, in which the authors introduced a logarithmic free energy to characterize the capillarity effect. However, their work did not give the details on how to reformulate the free energy and the corresponding energy inequality. Recently, 
the model proposed in  \cite{kou2022energy} \textcolor{black}{was reformulated} as a chemical potential-driven formulation, and the free energy was rewritten as a function of the wetting-phase saturation only. \textcolor{black}{On the other hand}, it should be noticed that the transformed model may exhibit deviations from physical reality since the wetting-phase pressure does not exist globally. \textcolor{black}{To retain the full physical  accuracy,} in this work we concentrate on the original model of two-phase flow in porous media \cite{gao2020thermodynamically,kou2022energy}. The model is based on \textcolor{black}{the coupling of} the mass conservation equation of binary fluids, Darcy's law, \textcolor{black}{the chemical potentials derived from the free energy,}
and the saturation constraint, which are given as follows:
\begin{align}	
	&\phi\frac{\partial  S_{\alpha} }{\partial t}+\nabla\cdot \textbf{u}_{\alpha}= q_{ \alpha } ,\label{e_modela} \\
	&\textbf{u}_{\alpha}=-\lambda _{\alpha}\textbf{K}\nabla(p + \mu_{\alpha} ), \label{e_modelb} \\
	&\mu_{\alpha} = \sigma_{\alpha} {\rm ln} (S_{\alpha}) + \sigma_{wn}(1-S_{\alpha}),\label{e_modeld}\\
	&S_{w}+S_{n}=1, \label{e_modelc}	
\end{align}
where $\alpha=w, n$. Here $\phi$ is the porosity of a porous medium and $S_{\alpha}$ is the saturation of phase $\alpha$. The wetting and non-wetting phases are denoted by the subscript $w$ and $n$, respectively. In fact, the pressure $p$ is introduced as a Lagrange multiplier with regard to the saturation constraint, $\textbf{u}_{\alpha}$ is Darcy's velocity and $q_{\alpha}$ is the injection/production rate of phase $\alpha$, and $\textbf{K}$ is the absolute permeability. Without loss of generality, it is assumed that $\textbf{K}$ is a symmetric and positive tensor. For simplicity, we further assume that $\textbf{K} = K \textbf{I}$, where $\textbf{I}$ is the identity matrix and $K$ is a positive real number. The phase mobility $\lambda_\alpha$ is defined as 
\begin{equation*} 
	\lambda_{\alpha} (S_{\alpha})  = \frac{ k_{r\alpha} (S_{\alpha}) }{ \eta_{\alpha} } >0 , \quad \alpha=w,n, 
\end{equation*}
where $k_{r \alpha}$ and $\eta_{\alpha}$ are the relative permeability and viscosity of phase $\alpha$, respectively. Specifically, the relative permeabilities are given by $	k_{ra}(S_{\alpha})=S_{\alpha}^{m},~ \alpha=w,n$, where $m$ is a positive constant. 

The equation of chemical potential \eqref{e_modeld} comes from the derivatives of the free energy density function $F$, as given by 
\begin{align*} 
	&\mu_{w}(S_{w},S_{n})=\frac{\partial F}{\partial S_{w}} = \sigma_{w} {\rm ln} (S_{w}) + \sigma_{wn}S_{n}, \\
	&\mu_{n}(S_{w},S_{n})=\frac{\partial F}{\partial S_{n}} = \sigma_{n} {\rm ln} (S_{n}) + \sigma_{wn}S_{w},
\end{align*}
where $F$ is expressed as
\begin{align*} 
	F(S_{w},S_{n})= \sum_{\alpha=w,n } \sigma _{\alpha }S_{\alpha }( {\rm ln} (S_{\alpha })-1)+   \sigma _{wn}S_{w}S_n,
\end{align*}
and $\sigma_{w}$, $\sigma_{n}$ and $\sigma_{wn}$ are positive energy parameters.

\textcolor{black}{Constructing numerical schemes for the two-phase flow model that are both high-order accurate and structure-preserving is notoriously challenging. The importance of this problem has attracted widespread attention, leading to the development of various numerical strategies \cite{hunt2014percolation} from different research groups. For instance, classical approaches such as the improved IMPES method have been explored for computational efficiency \cite{chen2004improved}, while other research, reviewed in works like \cite{joekar2012analysis}, utilizes dynamic pore-network models to clarify the underlying physics at both the pore and representative elementary volume scales. While many numerical methods have been proposed} 
\cite{CHEN2019641, ERN20101491, girault2021finite, hunt2014percolation},
\textcolor{black}{a central challenge lies in simultaneously satisfying several critical physical properties. Indeed, many established techniques often compromise on one or more key aspects such as unconditional energy stability, strict bound preservation, and mass conservation. For instance, the popular invariant energy quadratization (IEQ) \cite{yang2017numerical} or scalar auxiliary variable (SAV) \cite{shen2018scalar} approaches, while successful in achieving energy stability, frequently do so for a modified energy functional. This often compromises the original energy dissipation law and fails to preserve strict saturation bounds, particularly for higher-order schemes. These limitations are evident in recent works. The first-order IEQ-based scheme by Kou et al. \cite{kou2022energy}, for example, achieves a bound but is limited to first-order accuracy and requires a strict CFL condition. Similarly, the high-order schemes by Li et al. \cite{LI2025113864} based on a modified SAV approach preserve a modified energy dissipation  law and ensure global mass conservation, while the original energy dissipation or crucial local mass conservation is not satisfied. The difficulty in preserving these properties is not confined to a single class of methods; for instance, the discontinuous Galerkin (DG) formulations presented by W. Klieber and B. Riviere \cite{KLIEBER2006404} also struggle to maintain mass conservation for both phases. Furthermore, a significant challenge in many existing models is the high computational cost arising from the tight coupling between the saturation and the pressure. Consequently, the design of numerical schemes that are simultaneously high-order, fully-decoupled, bound-preserving, locally mass conservative, and original energy dissipative for this model remains a significant objective.} 

Therefore, inspired by the positivity-preserving approach developed in \cite{CiCPchen, CHEN2019100031, liu2021positivity} for the general dissipative systems, in this paper we propose the first- and second-order-in-time, fully-decoupled, structure-preserving numerical scheme,s based on the convex splitting approach in the temporal discretization and the cell centered finite difference spatial approximation. Our main contributions are listed as follows:
\begin{itemize}
	\item \textcolor{black}{We develop efficient, fully-decoupled, first- and second-order schemes for the thermodynamically consistent two-phase flow model. Distinct from existing approaches, these schemes are constructed by subtracting, rather than adding, the mass conservation equations of the two phases, thereby achieving a full decoupling. }
	
	\item \textcolor{black}{The constructed first-order scheme is bound-preserving and unconditionally uniquely solvable.} In terms of the second-order scheme, these properties hold under a parameter constraint $\sigma_{wn} < \frac{1}{2}(\sigma_{w}+\sigma_{n})$, where $\sigma_{w}$, $\sigma_{n}$ and $\sigma_{wn}$ denote the energy parameters, which ensures the strict convexity of the discrete energy functionals corresponding to the associated numerical schemes. Moreover, the constructed schemes satisfy local mass conservation and unconditionally obey the original energy dissipation law.
	
	\item A rigorous error estimate for the first-order scheme is established in the $\ell^{\infty}(0,T;  H_h^{-1} (\Omega))$\\$ \cap \ell^{2}(0,T; \ell^2(\Omega))$ norm for saturations of two phases, addressing a critical gap in the theoretical analysis of structure-preserving schemes for two-phase flow in porous media. 
	
\end{itemize}

\textcolor{black}{The primary contribution of this work is a novel numerical framework that, for the first time, unifies a complete set of crucial structure-preserving properties for the two-phase flow in porous media. Specifically, our framework yields both first- and second-order accurate schemes that are fully decoupled, while rigorously maintaining the original energy dissipation law, strict physical bounds on saturation, and local mass conservation. A further significant contribution is the rigorous convergence analysis for the first-order scheme, which represents the first such result for a scheme possessing such properties. While the reported schemes are based on a finite difference spatial discretization, the underlying theoretical principles are general, suggesting a clear pathway for extension to other spatial approximations, such as mass-lumped finite element or pseudo-spectral methods.}

The paper is organized as follows. In Section \ref{section2}, we review the physical properties of the model and some preliminary notations. In Section \ref{section3}, we propose the first- and second-order fully discrete numerical schemes with bound-preserving properties. In Section \ref{section4}, we present the proof of mass conservation and energy stability analysis. In Section \ref{section5}, a rigorous error analysis is provided for the first-order scheme. In Section \ref{section6}, a few numerical experiments are presented to validate the theoretical results. Finally, some concluding remarks are given in Section \ref{section7}.

\section{The physical properties of the model and some preliminaries}\label{section2}
Let $\Omega$ be a connected and smooth domain in space and the system is equipped with proper boundary conditions, dependent on the specified problems. We first introduce the total free energy 
\begin{align*}  
	\mathcal{E}_{tol}(t) = \int_{\Omega }\phi F(S_{w},S_{n}) d \textbf{x} , 
\end{align*}
and derive the corresponding energy dissipation law.
\medskip
\begin{theorem} \rm{\cite{kou2022energy, LI2025113864}} \label{thm_energy dissipation}
	The system \eqref{e_modela}-\eqref{e_modelc} satisfies
	\begin{align*}
		\frac{\partial \mathcal{E}_{tol}(t)}{\partial t} = & \sum_{\alpha =w,n} \int_{\Omega } (p+\mu_{ \alpha }) q_{ \alpha } d {\bf x}  
		- \sum_{\alpha =w,n}  \int_{\partial \Omega }  (p+\mu_{ \alpha }) \textbf{u}_{ \alpha } \cdot \textbf{n} ~ds  \\
		&- \sum_{\alpha =w,n}\int_{\Omega }\lambda _{\alpha  }  {K}  \left | \nabla (p+\mu _{\alpha }) \right |^{2} d {\bf x} .\nonumber
	\end{align*}
	In addition, assuming $q_{\alpha}=0$ and the no-flow boundary condition $ \textbf{u}_{\alpha} \cdot \textbf{n} =0 $, where $ \textbf{n} $ is the normal unit outward vector to $\partial \Omega$, we obtain the following energy dissipation law:
	\begin{align*}
		\frac{\partial \mathcal{E}_{tol}(t)}{\partial t} =-\sum_{\alpha =w,n}\int_{\Omega }\lambda _{\alpha  }{K}\left | \nabla (p+\mu _{\alpha }) \right |^{2}d {\bf x} \le0 .
	\end{align*}
\end{theorem}
The proof of the above theorem can be found in \cite{kou2022energy, LI2025113864}.
\subsection{Notations and discrete operators}\label{nota1}
The cell centered finite difference is taken as the spatial discretization. For simplification of presentation, we consider the two-dimensional rectangular domain $\Omega=[a_{1},b_{1}]\times[a_{2},b_{2}]$, while the three-dimensional case could be similarly formulated. For the sake of brevity, a uniform mesh of $\Omega$ is used with mesh size $h=\frac{b_{1}-a_{1}}{N}=\frac{b_{2}-a_{2}}{M}$, where $N$ and $M$ are positive integers. The cell points are denoted by $(x_{i},y_{i})$, where $x_{i}=a_{1}+ih,i=0,1,\cdots,N$, and $y_{j}=a_{2}+ih,j=0,1,\cdots,M$. We also introduce the staggered mesh points $x_{i+\frac{1}{2}}=x_{i}+\frac{1}{2}h$ and 
$y_{j+\frac{1}{2}}=y_{j}+\frac{1}{2}h$.
{\color{black}
	The discrete function spaces are defined as follows
	\begin{align*}
		&\mathcal{C}_{h}= \{ c:(x_{i+\frac{1}{2}},y_{j+\frac{1}{2}})\mapsto \mathbb{R},~ 0 \le i \le N-1,0 \le j \le M-1  \},\\
		&\mathcal{U}_{h}= \{ u:(x_{i},y_{j+\frac{1}{2}})\mapsto \mathbb{R},~0 \le i \le N, 0 \le j \le M-1  \} ,\\
		&\mathcal{V}_{h}= \{ v:(x_{i+\frac{1}{2}},y_{j})\mapsto \mathbb{R},~0 \le i \le N-1, 0 \le j \le M  \}.
	\end{align*}
	Here, for the components of a discrete function in the above space, we use the identification $c_{i+\frac{1}{2},j+\frac{1}{2}}=c(x_{i+\frac{1}{2}},y_{j+\frac{1}{2}})$. 
	The mean zero space is defined as 
	\begin{align*}
		\mathcal{C}_{h}^0= \{ c\in \mathcal{C}_h \mid 0=\bar{c}:=\frac{h^2}{\left| \Omega \right| }\sum_{i,j} c_{i+\frac{1}{2},j+\frac{1}{2}}  \}.
	\end{align*}
	For any $c \in \mathcal{C}_{h}$, the average operators are defined as 
	$$
	\Pi_x c_{i+1,j+\frac{1}{2}}= \frac{1}{2} (c_{i+\frac{1}{2},j+\frac{1}{2}} +c_{i+\frac{3}{2},j+\frac{1}{2}}),\quad \Pi_y c_{i+\frac{1}{2},j+1}= \frac{1}{2} (c_{i+\frac{1}{2},j+\frac{1}{2}} +c_{i+\frac{1}{2},j+\frac{3}{2}}).
	$$
	The discrete gradient operator is introduced as $\nabla_{h}=[\nabla_{h,x},\nabla_{h,y}]^{T}$, where $\nabla_{h,x}c $ and $\nabla_{h,y}c $ are given by 
	\begin{equation*}
		\aligned
		\nabla_{h,x}c_{i,j+\frac{1}{2}}&=\frac{c_{i+\frac{1}{2},j+\frac{1}{2}}-c_{i-\frac{1}{2},j+\frac{1}{2}}}{h},~~
		\nabla_{h,y}c_{i+\frac{1}{2},j}&=\frac{c_{i+\frac{1}{2},j+\frac{1}{2}}-c_{i+\frac{1}{2},j-\frac{1}{2}}}{h}.
		\endaligned
	\end{equation*}
	For $u \in \mathcal{U}_{h}$ and $v \in \mathcal{V}_{h}$, the discrete divergence operator becomes $\nabla_{h} \cdot [u,v]^{T}=D_{x}u+D_{y}v$, where $D_{x}$ and $D_{y}$ take the form as 
	\begin{equation*}
		\aligned
		D_{x}u_{i+\frac{1}{2},j+\frac{1}{2}}=\frac{u_{i+1,j+\frac{1}{2}}-u_{i,j+\frac{1}{2}}}{h},~~
		D_{y}v_{i+\frac{1}{2},j+\frac{1}{2}}=\frac{v_{i+\frac{1}{2},j+1}-v_{i+\frac{1}{2},j}}{h}.
		\endaligned
	\end{equation*}
	The discrete inner products and norms are introduced as follows:
	\begin{align*}	(c,c')_{h}=h^{2}\sum\limits_{i=0}^{N-1}\sum\limits_{j=0}^{M-1}c_{i+\frac{1}{2},j+\frac{1}{2}}c'_{i+\frac{1}{2},j+\frac{1}{2}}, ~c, c' \in \mathcal{C}_{h},\\
		(u,u')_{h}=h^{2}\sum\limits_{i=1}^{N-1}\sum\limits_{j=0}^{M-1}u_{i,j+\frac{1}{2}}u'_{i,j+\frac{1}{2}},~ u, u' \in \mathcal{U}_{h},\\	(v,v')_{h}=h^{2}\sum\limits_{i=0}^{N-1}\sum\limits_{j=1}^{M-1}v_{i+\frac{1}{2},j}v'_{i+\frac{1}{2},j},~ v, v' \in \mathcal{V}_{h}.
	\end{align*}
	For $\textbf{u}=[u,v]^{T}$ and $\textbf{u}'=[u',v']^{T}$, where $u,u' \in \mathcal{U}_{h}$ and $v,v' \in \mathcal{V}_{h}$, the inner-product becomes  
	$$(\textbf{u},\textbf{u}')_{h}=(u,u')_{h}+(v,v')_{h}.$$
	The discrete norms for $c \in \mathcal{C}_{h}, u \in \mathcal{U}_{h}^{0}$ and $v \in \mathcal{V}_{h}^{0}$ are defined as
	$$\left \|  c \right \| _{h}=(c,c)_{h}^{1/2},\quad  \left \|  u \right \| _{h}=(u,u)_{h}^{1/2},\quad  \left \|  v \right \| _{h}=(v,v)_{h}^{1/2},$$
	where the discrete function spaces $\mathcal{U}_{h}^{0}$ and $\mathcal{V}_{h}^{0}$ are the subsets of $\mathcal{U}_{h}$ and $\mathcal{V}_{h}$ with the homogeneous Neumann boundary condition. For $\textbf{u}=[u,v]^{T}$, where $ u \in \mathcal{U}_{h}^{0} $ and  $ v \in \mathcal{V}_{h}^{0} $, we define 
	$$\left \|  \textbf{u}  \right \|_{h}^{2} =\left \|  u \right \| _{h}^{2}+\left \|  v \right \| _{h}^{2},$$which implies that 
	$$\left \|  \nabla_{h}c \right \| _{h}^{2}=\left \|  \nabla_{h,x}c \right \| _{h}^{2}+\left \|  \nabla_{h,y}c \right \| _{h}^{2}, \ \ c \in \mathcal{C}_{h}.$$
	Therefore, the following discrete  summation-by-parts formulas could be easily obtained 
	\begin{align}
		(u,\nabla_{h,x}c)_{h}=-(D_{x}u,c)_{h}, \ \ u \in \mathcal{U}_{h}^{0},~ c \in \mathcal{C}_{h},\label{4.1a} \\	(v,\nabla_{h,y}c)_{h}=-(D_{y}v,c)_{h}, \ \ v \in \mathcal{V}_{h}^{0},~ c \in \mathcal{C}_{h}.\label{4.1b}
	\end{align} 
	As a direct application of (\ref{4.1a}) and (\ref{4.1b}), we see that 
	\begin{equation*}
		(\nabla_{h}\cdot \textbf{u},c)_{h}=-(\textbf{u},\nabla_{h}c)_{h},
	\end{equation*}
	where $\textbf{u}=[u,v]^{T}$, $u \in \mathcal{U}_{h}^{0}$, $v \in \mathcal{V}_{h}^{0}$, and $c \in \mathcal{C}_{h}$.}

\section{Fully discrete numerical schemes and bound-preserving properties}\label{section3} In this section, we present first- and second-order fully-discrete numerical schemes with a theoretical analysis of the unique solvability and bound-preserving properties.

\subsection{The first-order scheme}
The first-order fully-discrete scheme is proposed as follows: 
\begin{align}
	&\phi_h \frac{S_{w,h}^{k+1}-S_{w,h}^{k}}{\Delta t} - \nabla_h \cdot \left( \check{M}_{w,1}  \nabla_h(p_h^{k+1} + \mu_{w,h}^{k+1})    \right) = 0,\label{first_scheme1}\\
	&\phi_h \frac{S_{n,h}^{k+1}-S_{n,h}^{k}}{\Delta t} - \nabla_h \cdot \left( \check{M}_{n,1}  \nabla_h(p_h^{k+1} + \mu_{n,h}^{k+1})    \right) = 0,\label{first_scheme2}\\
	&\mu_{w,h}^{k+1}= \sigma_{w} {\rm ln}(S_{w,h}^{k+1} ) + \sigma_{wn} S_{n,h}^{k},\label{first_scheme3}\\
	&\mu_{n,h}^{k+1}= \sigma_{n} {\rm ln}(S_{n,h}^{k+1} ) + \sigma_{wn} S_{w,h}^{k},\label{first_scheme4}\\
	&S_{n,h}^{k+1}=1-S_{w,h}^{k+1},\label{first_scheme5}
\end{align}
where the mobility functions are given by $\check{M}_{\alpha,1}=\frac{ ( {S}_{\alpha,h}^{k} )^m }{\eta_{\alpha}} K$ for $\alpha=w,n$, namely, $\check{M}_{\alpha,1,i+1,j+\frac{1}{2}}=\Pi_x({\lambda}^{k}_{\alpha} K)$ and $\check{M}_{\alpha,1,i+\frac{1}{2},j+1}=\Pi_y({\lambda}^{k}_{\alpha} K)$.
Here $\phi_h$ depends on space, while it is time-independent, and satisfies $0< \phi_h < 1$.
The operators are defined as $\mathcal{L}_{\check{M}_{\alpha,1}}=(- \nabla_h \cdot  \check{M}_{\alpha,1}  \nabla_h )/\phi_h $.

It is obvious that \eqref{first_scheme1}-\eqref{first_scheme5} are coupled nonlinear system for $p^{k+1}_h$ and $S^{k+1}_{w,h}$. We eliminate the Lagrange multiplier $p^{k+1}_h$ by applying the operators $\mathcal{L}^{-1}_{\check{M}_{w,1}}$ and $\mathcal{L}^{-1}_{\check{M}_{n,1}}$ to  \eqref{first_scheme1} and \eqref{first_scheme2} respectively, then subtract the two equations. This process in turn yields the following equivalent fully decoupled numerical scheme. 

\medskip
\textbf{Efficient implementation:}
\textit{Step 1}: Given $S_{w,h}^k$ and $S_{n,h}^k$, find $S_{w,h}^{k+1}$ such that
\begin{align}
	&\mu_{w,h}^{k+1}-\mu_{n,h}^{k+1}+ \frac{1}{\Delta t} \mathcal{L}^{-1}_{\check{M}_{w,1}} ( S_{w,h}^{k+1}-S_{w,h}^{k} ) + \frac{1}{\Delta t} \mathcal{L}^{-1}_{\check{M}_{n,1}} (S_{w,h}^{k+1}-S_{w,h}^{k}) =0,\label{refirst_w}\\
	&\mu_{w,h}^{k+1}= \sigma_{w} {\rm ln}(S_{w,h}^{k+1} ) + \sigma_{wn} S_{n,h}^{k},\label{refirst_scheme_muw}\\
	&\mu_{n,h}^{k+1}= \sigma_{n} {\rm ln}(1-S_{w,h}^{k+1} ) + \sigma_{wn} S_{w,h}^{k}.\label{refirst_scheme_mun}
\end{align}
\textit{Step 2}: Given $S_{w,h}^{k+1}$, find $p_{h}^{k+1}$ and $S_{n,h}^{k+1}$ such that
\begin{align}
	&p_h^{k+1} + \mu_{w,h}^{k+1} + \frac{1}{\Delta t} \mathcal{L}^{-1}_{\check{M}_{w,1}} ( S_{w,h}^{k+1}-S_{w,h}^{k} ) = 0,\label{refirst_p}\\
	&S_{n,h}^{k+1}=1-S_{w,h}^{k+1}.\label{refirst_n}
\end{align}	

\subsection{The second-order scheme}
The second-order fully-discrete scheme is proposed as follows:	
\begin{align}
	&\phi_h \frac{S_{w,h}^{k+1}- S_{w,h}^{k}}{\Delta t} - \nabla_h \cdot \Big(\check{M}_{w,2}  \nabla_h(p_h^{k+\frac{1}{2}} + \mu_{w,h}^{k+\frac{1}{2}}) \Big) = 0, \label{sec_scheme1} \\
	&\phi_h \frac{S_{n,h}^{k+1}-S_{n,h}^{k}}{\Delta t} - \nabla_h \cdot \Big( \check{M}_{n,2}  \nabla_h(p_h^{k+\frac{1}{2}} + \mu_{n,h}^{k+\frac{1}{2}}) \Big) = 0, \label{sec_scheme2} \\
	&\mu_{w,h}^{k+\frac{1}{2}} = \sigma_{w} \Big( \mathcal{H}(S_{w,h}^{k+1}, S_{w,h}^k) - 1 \Big) + \sigma_{wn} S_{n,h}^{k+\frac{1}{2}} + \Delta t (\ln S_{w,h}^{k+1} - \ln S_{w,h}^{k}), \label{sec_scheme3} \\
	&\mu_{n,h}^{k+\frac{1}{2}} = \sigma_{n} \Big( \mathcal{H}(S_{n,h}^{k+1}, S_{n,h}^k) - 1 \Big) + \sigma_{wn} S_{w,h}^{k+\frac{1}{2}} + \Delta t (\ln S_{n,h}^{k+1} - \ln S_{n,h}^{k}), \label{sec_scheme4} \\
	&S_{n,h}^{k+1} = 1-S_{w,h}^{k+1} , \label{sec_scheme5}
\end{align}
where $${\mathcal{H}(S_{\alpha,h}^{k+1}, S_{\alpha,h}^k)=\frac{ S_{\alpha,h}^{k+1}{\rm ln} (S_{\alpha,h}^{k+1})- S_{\alpha,h}^{k}{\rm ln} (S_{\alpha,h}^{k})}  { S_{\alpha,h}^{k+1} - S_{\alpha,h}^k }},$$
and the mobility functions are given by $\check{M}_{\alpha,2}=\frac{ ( \tilde{S}_{\alpha,h}^{k+1} )^m }{\eta_{\alpha}} K$ for $\alpha=w,n$, namely, $\check{M}_{\alpha,2,i+1,j+\frac{1}{2}}=\Pi_x(\tilde{\lambda}^{k}_{\alpha} K)$ and $\check{M}_{\alpha,2,i+\frac{1}{2},j+1}=\Pi_y(\tilde{\lambda}^{k}_{\alpha} K)$.
The term $\tilde{S}_{\alpha,h}^{k+1}$ represents a second-order explicit extrapolation of the saturation, defined as $\tilde{S}_{\alpha,h}^{k+1} = \frac{3}{2} S_{\alpha,h}^{k}- \frac{1}{2} S_{\alpha,h}^{k-1}$.
Finally, for convenience, we define the discrete operators $\mathcal{L}_{\check{M}_{\alpha,2}}=(- \nabla_h \cdot  \check{M}_{\alpha,2}  \nabla_h )/\phi_h $.

Analogous to the first-order scheme, an equivalent and fully decoupled numerical system could be derived as follows.

\medskip	
\textbf{Efficient implementation:}
\textit{Step 1}: Given $S_{w,h}^k$, $S_{w,h}^{k-1}$, $S_{n,h}^k$ and $S_{n,h}^{k-1}$, find $S_{w,h}^{k+1}$ such that
\begin{align}
	&\mu_{w,h}^{k+\frac{1}{2}} - \mu_{n,h}^{k+\frac{1}{2}} +  \frac{1}{\Delta t} \mathcal{L}^{-1}_{\check{M}_{w,2}} ( S_{w,h}^{k+1}-S_{w,h}^{k} )  +\frac{1}{\Delta t}\mathcal{L}^{-1}_{\check{M}_{n,2}} (S_{w,h}^{k+1}-S_{w,h}^{k})=0,\label{resecond_w}\\
	&\mu_{w,h}^{k+\frac{1}{2}}= \sigma_{w} \Big(\mathcal{H}(S_{w,h}^{k+1}, S_{w,h}^k) -1 \Big) + \sigma_{wn} (1-S_{w,h}^{k+\frac{1}{2}}) + \Delta t ({\rm ln} S_{w,h}^{k+1} - {\rm ln} S_{w,h}^{k}),\label{resecond_muw}\\
	&\mu_{n,h}^{k+\frac{1}{2}}= \sigma_{n} \Big(\mathcal{H}(1-S_{w,h}^{k+1}, 1-S_{w,h}^k) -1 \Big) + \sigma_{wn} S_{w,h}^{k+\frac{1}{2}} \label{resecond_mun}\\
	& \quad\quad\quad+ \Delta t ({\rm ln} (1-S_{w,h}^{k+1}) - {\rm ln} S_{n,h}^{k}).\nonumber
\end{align}
\textit{Step 2}: Given $S_{w,h}^{k+1}$, find $p_h^{k+1}$ and $S_{n,h}^{k+1}$ such that
\begin{align}
	&p_h^{k+\frac{1}{2}} + \mu_{w,h}^{k+\frac{1}{2}}+  \frac{1}{\Delta t} \mathcal{L}^{-1}_{\check{M}_{w,2}} ( S_{w,h}^{k+1}-S_{w,h}^{k} ) = 0,\label{resecond_p}\\
	&S_{n,h}^{k+1} =1-S_{w,h}^{k+1}.\label{resecond_n}
\end{align}

\begin{rem}
	In the second-order scheme, to ensure both the second-order temporal accuracy and a point-wise positive lower bound, the following formulas could be used
	\begin{align}
		\check{M}_{w,2,i+1,j+\frac{1}{2}}=(\Pi^2_x(\tilde{\lambda}^{k}_{w}K)_{i+1,j+\frac{1}{2}} + \Delta t^6)^{1/2},\nonumber\\	\check{M}_{w,2,i+\frac{1}{2},j+1}=(\Pi^2_y(\tilde{\lambda}^{k}_{w}K)_{i+\frac{1}{2},j+1} + \Delta t^6)^{1/2},\nonumber
	\end{align}
	to avoid a numerical instability. An analogous technique is applicable to $\check{M}_{n,2}$.
\end{rem}
\medskip
\begin{rem}
	In the case of $S_{w,h}^{k+1}=S_{w,h}^k$, we have the following definition for $\mu_{\alpha,h}$ in \eqref{resecond_muw} and \eqref{resecond_mun}, $\alpha=w,n$:  
	\begin{equation}
		\mu_{w,h}^{k+\frac{1}{2}} =
		\begin{cases} 
			\sigma_{w} ( {\mathcal{H}(S_{w,h}^{k+1}, S_{w,h}^k) }- 1 ) 
			+ \sigma_{wn}  S_{n,h}^{k+\frac{1}{2}} 
			+ \Delta t (\ln S_{w,h}^{k+1} - \ln S_{w,h}^{k}), &\hspace{-0.15cm} S_{w,h}^{k+1} \neq S_{w,h}^{k}, \\
			\sigma_{w}  \ln(S_{w,h}^{k+\frac{1}{2}})
			+ \sigma_{wn} S_{n,h}^{k+\frac{1}{2}}
			+ \Delta t (\ln S_{w,h}^{k+1} - \ln S_{w,h}^{k}), &\hspace{-0.15cm} S_{w,h}^{k+1} = S_{w,h}^{k},
		\end{cases}\nonumber
	\end{equation}
	where $S_{n,h}^{k+\frac{1}{2}}= 1-S_{w,h}^{k+\frac{1}{2}}$. A similar definition could be made to $\mu_{n,h}^{k+\frac{1}{2}}$.
\end{rem}

\medskip
\begin{rem}
	Due to the unbiased nature of the wetting and non-wetting phases, equations \eqref{refirst_w} and \eqref{resecond_w} can be equivalently reformulated to solve for the non-wetting phase saturation $S_{n,h}^{k+1}$.
\end{rem}
\medskip
Next, we present the theoretical analysis that the above numerical schemes have unique solutions, and the bound of unknowns $S_{\alpha}$, $\alpha=w,n$, is ensured, under homogeneous injection/production rates and boundary conditions. If the numerical solution to \eqref{refirst_w}-\eqref{refirst_n} exists, it is easy to see that
\begin{align*}
	\bar{S}_{w,h}^{k+1}= \bar{S}_{w,h}^{k}=\cdots =\bar{S}_{w,h}^0 = \beta_0, ~~\forall k \in  \mathbb{ N } .
\end{align*}
To derive the desired results, we first recall the following preliminary estimates \cite{CHEN2019100031} to continue the unique solvability and bound-preserving analysis for the constructed schemes. For any $\varphi \in  \mathcal{C}_{h}$, there exists a unique $\psi \in \mathcal{C}^0_{h}$ that solves
\begin{align*}
	\mathcal{L}_{D}(\psi) := -\nabla_h \cdot (D \nabla_h \psi) = \varphi - \overline{\varphi}, ~~\overline{\varphi} := |\Omega|^{-1}  (\varphi, 1)_{h}.
\end{align*}
Moreover, a bilinear form is defined for any $\varphi_1, \varphi_2 \in \mathcal{C}^0_{h}$ as follows
\begin{align*}
	\langle \varphi_1, \varphi_2 \rangle_{\mathcal{L}_D^{-1}} := (D \nabla_h \psi_1, \nabla_h \psi_2)_{h}, 
\end{align*}
where $\psi_i \in \mathcal{C}^0_{h}$ is the unique solution to
\begin{align*}
	\mathcal{L}_D(\psi_i) := -\nabla_h \cdot (D \nabla_h \psi_i) = \varphi_i, \quad i = 1, 2. 
\end{align*}
The following identity turns out to be straightforward: 
\begin{align*}
	\langle \varphi_1, \varphi_2 \rangle_{\mathcal{L}_D^{-1}} = ( \varphi_1, \mathcal{L}_D^{-1}(\varphi_2) )_{h} = ( \mathcal{L}_D^{-1}(\varphi_1), \varphi_2 )_{h}, 
\end{align*}
and since $\mathcal{L}_D$ is symmetric positive definite and $\langle \cdot, \cdot \rangle_{L_D^{-1}}$ is an inner product on $\mathcal{C}^0_{h}$ \cite{wcsina2011}. In general, the norm associated to this inner product is denoted as $\|\varphi\|_{L_D^{-1}} := \sqrt{\langle \varphi, \varphi \rangle_{L_D^{-1}}}$, for all $\varphi \in \mathcal{C}^0_{h}$. In particular, if $D \equiv 1$, the discrete $H^{-1}$ inner product and norm are defined as 
\begin{equation} 
	\begin{aligned} 
		& 
		\langle \varphi_1 , \varphi_2 \rangle_{-1,h} = ( \varphi_1 , (-\Delta_h)^{-1} \varphi_2 )_h  
		= ( (-\Delta_h)^{-1} \varphi_1, \varphi_2 )_h ,  \quad \forall  \varphi_1, \varphi_2 \in \mathcal{C}^0_{h} , 
		\\
		& 
		\| \varphi \|^2_{-1,h} = \langle \varphi , \varphi \rangle_{-1,h} , \quad 
		\forall  \varphi \in \mathcal{C}^0_h . 
	\end{aligned} 
	\label{discrete H-1} 
\end{equation} 


\medskip
\begin{lemma}\label{lemL}
	\rm{	\cite{CHEN2019100031}} Suppose that $\varphi_1, \varphi_2 \in \mathcal{C}_{h}$, with $\varphi_1 - \varphi_2 \in \mathcal{C}^0_{h}$. Assume that $\|\varphi_1\|_{\infty}, \|\varphi_2\|_{\infty} \le \mathcal{M}_h$, with a constant $\mathcal{M}_h > 0$, and $\check{\mathcal{M}} \geq {\mathcal{M}}_0$ at a point-wise level, for some constant $\mathcal{M}_0 > 0$ that is independent of $h$. Then we have the following estimate:
	\begin{equation*}
		\left \| 	\mathcal{L}_{\check{\mathcal{M}}}^{-1}(\varphi_1 - \varphi_2)\right\|_{\infty}   \leq C_2 := \widetilde{C}_2 \mathcal{M}_0^{-1}h^{-1/2},
	\end{equation*}
	where $\widetilde{C}_2 > 0$ depends only upon $\mathcal{M}_h$ and  $\Omega$.
\end{lemma}

\medskip
To facilitate the subsequent analysis, we introduce the following three smooth functions \cite{CiCPchen}, for any fixed $a > 0$: 
\begin{align*}
	H_a^1(x)&:=\frac{F(x)-F(a)}{x - a},\quad\forall x>0,\\
	H_a^0(x)&:=\int_{a}^{x}H_a^1(t)dt=\int_{a}^{x}\frac{F(t)-F(a)}{t - a}dt, \quad\forall x>0,\\
	H_a^2(x)&:=(H_a^1)^\prime(x)=\frac{F^\prime(x)(x - a)-(F(x)-F(a))}{(x - a)^2}=(H_a^0)^{\prime\prime}(x), \quad\forall x>0,
\end{align*}
where $F(x)=x {\rm ln} x$. 

\medskip
\begin{lemma}\label{lemG} \rm{\cite{CiCPchen}} Let \(a>0\) be fixed, we have
	\begin{enumerate}
		\renewcommand{\theenumi}{(\roman{enumi})}
		\item $H_a^2(x) \geq 0$ and $H_a^0(x)$ is convex, as a function of $x$, for $x > 0$.
		\item  $(H_a^1)^\prime(x)=\frac{1}{2\xi}$, for some $\xi$ between $a$ and $x$.
	\end{enumerate}
\end{lemma}

The strict convexity of the discrete energy functional guarantees a unique minimizer, which coincides with the numerical solution. The proof of Theorem \ref{thmsec} (regarding positivity preservation and unique solvability) relies on a parameter constraint below.

\begin{lemma}\label{energypa}
	Suppose that the energy parameters satisfy  $\sigma_{wn} < \frac{1}{2}({\sigma_{w}}+{\sigma_{n}})$ in \eqref{eq22}, then the following inequality holds true for any value of $ x \in (0,1)$,
	\begin{align}
		\frac{d^2}{d {x}^2}( \sigma_{w} H^0_{a}(x) + \sigma_{n} H^0_{1-a}(1-x) -  \frac{ \sigma_{wn} }{2}x^2 )>0 ,~\forall~ 0<a<1.\label{eq22}
	\end{align}
	\begin{proof}
		By recalling property (ii) in Lemma \ref{lemG}, we are able to equivalently rewrite \eqref{eq22} as
		\begin{align}\label{eqtras}
			\sigma_{wn} <( \sigma_{w}  H^2_{a}(x) + \sigma_{n} H^2_{1-a}(1-x))_{min}=(\sigma_{w}\frac{1}{2\xi_1}+\sigma_{n} \frac{1}{2\xi_2})_{min}, 
		\end{align}
		where $\xi_1 \in (0,1)$ between $a$ and $x$, $\xi_2 \in (0,1)$ between $1-a$  and  $1-x$. Specifically, since $\frac{1}{\xi_1}>1$ and $\frac{1}{\xi_2}>1$, we see that $\sigma_{w} \cdot \frac{1}{2\xi_1}+\sigma_{n} \cdot \frac{1}{2\xi_2} > \frac{1}{2} (\sigma_{w}+\sigma_{n}).$ Therefore, $ \frac{1}{2} (\sigma_{w}+\sigma_{n})$ serves as a lower bound for the right-hand side of \eqref{eqtras}. This leads to the loose constraint $\sigma_{wn} < \frac{1}{2}({\sigma_{w}}+{\sigma_{n}})$ in \eqref{eq22}. The proof is completed. 
	\end{proof}
\end{lemma}

\subsection{Unique solvability and bound-preserving analysis of the first-order scheme}
The bound-preserving and unique solvability properties of the first-order scheme \eqref{refirst_w}-\eqref{refirst_n} are established below in the three dimensional case. This analytical framework could be naturally extended to the two dimensional case.

\medskip
\begin{theorem}\label{thmfirst}
	Given $S_{\alpha,h}^k \in \mathcal{C}_h$, with $0 <S_{\alpha,h}^k < 1$ at a point-wise level. There exists a unique solution $S_{\alpha,h}^{k+1} \in \mathcal{C}_h$ to the first-order scheme \eqref{refirst_w}-\eqref{refirst_n}, with $0 <S_{\alpha,h}^{k+1} < 1$, $\alpha=w,n$. 
\end{theorem}

\medskip
\begin{proof}
	The numerical solution of \eqref{refirst_w} is equivalent to the minimizer of the following discrete energy functional:
	\begin{align*}
		\mathcal{J}^k(S_w) &= \frac{1}{2 \Delta t}   \left \| S_w-S_{w,h}^{k} \right \|^2_{\mathcal{L}^{-1}_{\check{M}_{w,1}} }   +  \frac{1}{2 \Delta t} \left \| S_w-S_{w,h}^{k} \right \|^2_{\mathcal{L}^{-1}_{\check{M}_{n,1}} } +   \sigma_{w} (S_w,{\rm ln}(S_w) -1 )_h \nonumber\\&
		+ 	\sigma_{n} ((1-S_w),{\rm ln}(1-S_w ) -1 )_h  + \sigma_{wn} (1-2S_{w,h}^{k},S_w)_h,
	\end{align*}
	over an admissible set 
	\begin{align*}
		\Omega_{H}=\{ S_w \in \mathcal{C}_{h}~ |~ 0 \le S_w \le 1, (S_w-\bar{S}_{w,h}^0,1)_h=0 \}.
	\end{align*}
	It is clear that $\mathcal{J}^k(S_w)$ is a strictly convex function over this set.

Next, this minimization problem is transformed into an equivalent one, for the convenience of subsequent analysis:
\begin{equation}
	\begin{aligned}\label{first_convex}
		\mathcal{F}^k(\varphi):&= \mathcal{J}^k(\varphi+\bar{S}_{w,h}^0)\\
		&= \frac{1}{2 \Delta t}   \left \| \varphi+\bar{S}_{w,h}^0-S_{w,h}^{k} \right \|^2_{\mathcal{L}^{-1}_{\check{M}_{w,1}} }   +  \frac{1}{2 \Delta t} \left \| \varphi+\bar{S}_{w,h}^0-S_{w,h}^{k} \right \|^2_{\mathcal{L}^{-1}_{\check{M}_{n,1}} } \\
		&+   \sigma_{w} (\varphi+\bar{S}_{w,h}^0,{\rm ln}(\varphi+\bar{S}_{w,h}^0 ) -1)_h+ \sigma_{wn} (\varphi+\bar{S}_{w,h}^0,1-2S_{w,h}^{k})_h\\
		&+ 	\sigma_{n} ((1-(\varphi+\bar{S}_{w,h}^0)),{\rm ln}(1-(\varphi+\bar{S}_{w,h}^0) ) -1)_h ,
	\end{aligned}
\end{equation}
defined on the set
\begin{align*}
	\Omega_{H}^0=\{  \varphi \in \mathcal{C}^0_{h} ~|~ -\bar{S}_{w,h}^0 \le \varphi \le 1- \bar{S}_{w,h}^0 \} . 
\end{align*}
If $\varphi \in \Omega_{H}^0$ minimizes $\mathcal{F}^k$, then $S_w := \varphi + \bar{S}_{w,h}^0 \in \Omega_{H}$ minimizes $\mathcal{J}^k$, and vice versa. Subsequently, we aim to prove that there exists a minimizer of $\mathcal{F}^k$ over the domain $\Omega_{H}^0$. The following closed domain is taken into account: 
\begin{align*}
	\Omega_{H,\delta}^0=\{  \varphi \in \mathcal{C}^0_{h} ~|~ \delta -\bar{S}_{w,h}^0 \le \varphi \le 1- \bar{S}_{w,h}^0-\delta \} , \quad \mbox{for} \, \, \delta \in (0, \frac{1}{2}) . 
\end{align*}
Since $\Omega_{H,\delta}^0$ is a bounded, compact, and convex set, the continuous function  $\mathcal{F}^k$ has a minimizer over $\Omega_{H,\delta}^0$. The key point of the bound-preserving analysis is that such a minimizer could not occur on the boundary of $\Omega_{H,\delta}^0$, if $\delta$ is sufficiently small. To be more explicit, by the boundary of $\Omega_{H,\delta}^0$, we refer to the points of $\psi+ \bar{S}_{w,h}^0 =1-\delta$ or $\psi+ \bar{S}_{w,h}^0 =\delta$, precisely.

To get a contradiction, suppose that the minimizer of $\mathcal{F}^k$, call it $\varphi^{\star}$, occurs at a boundary point of $\Omega_{H,\delta}^0$. First, let us assume that $\varphi^{\star}_{\vec{\alpha}_0}+ \bar{S}_{w,h}^0=\delta$, so that the grid function $\varphi^{\star}$ has a global minimum at $\vec{\alpha}_0$. Let $\vec{\alpha}_{1}$ be a grid point where $\varphi^{*}$ attains its maximum value.  By the fact that $\bar{\varphi}^{\star}=0$, it is obvious that 
$$ 1-\delta \ge \varphi^{\star}_{\vec{\alpha}_1} + \bar
{S}_{w,h}^0 \ge  \bar{S}_{w,h}^0.$$
Since $\mathcal{F}^k$ is smooth over $\Omega_{H,\delta}^0$, for all $\varphi \in \mathcal{C}^0_{h}$, the directional derivative turns out to be 
\begin{equation}
	\begin{aligned}
		&d_s\mathcal{F}^k(\varphi^{\star}+s\psi)|_{s=0}
		= \frac{1}{ \Delta t}  (\mathcal{L}^{-1}_{\check{M}_{w,1}}( \varphi^{\star}+\bar{S}_{w,h}^0 -S_{w,h}^{k}) , \psi)_h  \\
		&\quad\quad+  \frac{1}{ \Delta t} (\mathcal{L}^{-1}_{\check{M}_{n,1}}( \varphi^{\star}+\bar{S}_{w,h}^0 -S_{w,h}^{k} ) , \psi )_h  +   (\sigma_{w} {\rm ln}( \varphi^{\star}+\bar{S}_{w,h}^0 ),\psi )_h \\
		&\quad\quad-   (\sigma_{n} {\rm ln}(1-\varphi^{\star} - \bar{S}_{w,h}^0 ),\psi)_h   +   (\sigma_{wn} (1-2S_{w,h}^{k} ),\psi )_h .
	\end{aligned}
\end{equation}
Let us pick the direction $\psi \in \mathcal{C}_h^0$, such that $ \psi_{i,j,k}=\delta_{i,i_0}\delta_{j,j_0} \delta_{k,k_0}-  \delta_{i,i_1}\delta_{j,j_1} \delta_{k,k_1}$. In turn, the derivative may be expressed as
\begin{align}
	&\frac{1}{h^3}d_s\mathcal{F}^k(\varphi^{\star}+s\psi)|_{s=0}\notag\\
	= &\frac{1}{ \Delta t}   \mathcal{L}^{-1}_{\check{M}_{w,1}}( \varphi^{\star}+\bar{S}_{w,h}^0 -S_{w,h}^{k})_{\vec{\alpha}_0 }  - \frac{1}{ \Delta t}   \mathcal{L}^{-1}_{\check{M}_{w,1}}( \varphi^{\star}+\bar{S}_{w,h}^0 -S_{w,h}^{k})_{\vec{\alpha}_1} \nonumber\\
	&+ \frac{1}{ \Delta t}   \mathcal{L}^{-1}_{\check{M}_{n,1}}( \varphi^{\star}+\bar{S}_{w,h}^0 -S_{w,h}^{k})_{\vec{\alpha}_0}  - \frac{1}{ \Delta t}   \mathcal{L}^{-1}_{\check{M}_{n,1}}( \varphi^{\star}+\bar{S}_{w,h}^0 -S_{w,h}^{k})_{\vec{\alpha}_1} \\
	&+   \sigma_{w} {\rm ln}( \varphi^{\star}_{\vec{\alpha}_0}+\bar{S}_{w,h}^0 ) - \sigma_{w} {\rm ln}( \varphi^{\star}_{\vec{\alpha}_1}+\bar{S}_{w,h}^0 ) - \sigma_{n} {\rm ln}(1-\varphi^{\star}_{\vec{\alpha}_0} - \bar{S}_{w,h}^0 )\nonumber\\
	& +\sigma_{n} {\rm ln}(1-\varphi^{\star}_{\vec{\alpha}_1} - \bar{S}_{w,h}^0 )+  \sigma_{wn} (-2 S_{w,h,\vec{\alpha}_0}^{k} 	+ 2 S_{w,h,\vec{\alpha}_1}^{k} ).\nonumber
\end{align}
Let $S_w^{\star}:=\varphi^{\star}+\bar{S}_{w,h}^0$. Since $S_{w,\vec{\alpha}_0}^{\star}=\delta$,  $ \bar{S}_{w,h}^0 \le S^{\star}_{w,\vec{\alpha}_1} \le 1-\delta$, we see that 
\begin{align}
	&\sigma_{w} {\rm ln}( S^{\star}_{w,\vec{\alpha}_0}) - \sigma_{w} {\rm ln}( S^{\star}_{w,\vec{\alpha}_1} )- \sigma_{n} {\rm ln}(1-S^{\star}_{w,\vec{\alpha}_0}  ) +\sigma_{n} {\rm ln}(1-S^{\star}_{w,\vec{\alpha}_1}  ) \nonumber\\ 
	\le&\sigma_{w} {\rm ln} \delta -\sigma_{n} {\rm ln}(1-\delta) -\sigma_{w}   {\rm ln}  ( \bar{S}_{w,h}^0 ) + \sigma_{n}   {\rm ln}(1-\bar{S}_{w,h}^0) \\
	=&\sigma_{w} {\rm ln} \frac{\delta}{ \bar{S}_{w,h}^0 } - \sigma_{n} {\rm ln} \frac{1-\delta}{ 1-\bar{S}_{w,h}^0 }.\nonumber
\end{align} 
In terms of the numerical solution $S_{w,h}^k$ at the previous time step, the a priori assumption $\|  S_{w,h}^k \|_\infty \le 1$ indicates that
\begin{align*}
	-2 \le	S_{w,h,\vec{\alpha}_1}^{k}  - S_{w,h,\vec{\alpha}_0}^{k}  \le 2.
\end{align*}
An application of Lemma \ref{lemL} yields
\begin{align*}
	-4C_2 \le A_w(\omega_{\vec{\alpha}_0}) - A_w(\omega_{\vec{\alpha}_1}) 
	+ A_n(\omega_{\vec{\alpha}_0}) - A_n(\omega_{\vec{\alpha}_1}) \le 4C_2,
\end{align*}
where $A_w(\omega_{\vec{\alpha}}) = \mathcal{L}^{-1}_{\check{M}_{w,1}}(\omega_{\vec{\alpha}})$, $ A_n(\omega_{\vec{\alpha}}) = \mathcal{L}^{-1}_{\check{M}_{n,1}}(\omega_{\vec{\alpha}})$, $\omega_{\vec{\alpha}} = (S_w^{\star} - S_{w,h}^k)_{\vec{\alpha}}$.

Consequently, the following bound becomes available: 
\begin{align*}
	\frac{1}{h^3}d_s\mathcal{F}^k(\varphi^{\star}+s\psi)|_{s=0} \le \sigma_{w} {\rm ln} \frac{\delta}{ \bar{S}_{w,h}^0 } - \sigma_{n} {\rm ln} \frac{1-\delta}{ 1-\bar{S}_{w,h}^0 }  +  4\sigma_{wn} + 4C_2 \Delta t^{-1}.
\end{align*}
We denote $C_3= 4\sigma_{wn} + 4C_2 \Delta t^{-1}$. Notice that $C_3$ is a constant for a fixed $\Delta t$ though it becomes singular as $\Delta t \to 0 $. For any fixed $\Delta t$, we may choose $\delta$ sufficiently small so that
\begin{align}\label{delta_re}
	\frac{1}{h^3}d_s\mathcal{F}^k(\varphi^{\star}+s\psi)|_{s=0} \le \sigma_{w} {\rm ln} \frac{\delta}{ \bar{S}_{w,h}^0 } - \sigma_{n} {\rm ln} \frac{1-\delta}{ 1-\bar{S}_{w,h}^0 }  +C_3 <0.
\end{align}   
In turn, the following inequality become valid, provided that $\delta$ satisfies \eqref{delta_re}: 
\begin{align*} \frac{1}{h^3}d_s\mathcal{F}^k(\varphi^{\star}+s\psi)|_{s=0} < 0.
\end{align*}
This contradicts the assumption that $\mathcal{F}^k$ reaches its minimum at $\varphi^{\star}$, since the directional derivative is negative in a direction that points towards the interior of $\Omega^0_{H,\delta}$. Similar to the above argument, we are also able to prove that the global minimum of $\mathcal{F}^k$ on $\Omega^0_{H,\delta}$ cannot occur at the boundary point $\varphi^{\star}$ such that $\varphi^{\star}_{\vec{\alpha}_0} + \bar{S}_{w,h}^0 = 1-\delta$, for some $\vec{\alpha}_0$. Therefore, the global minimum of $\mathcal{F}^k$ over $\Omega^0_{H,\delta}$ could only possibly occur at an interior point, for $\delta>0$  sufficiently small. Then we conclude that there must be a solution $S_w=\varphi+\bar{S}_{w,h}^0  \in \Omega_{H}$ that minimizes $\mathcal{J}^k$. The existence of the numerical solution is proved.
Meanwhile, since $\mathcal{J}^k$ is a strictly convex function over $\Omega_{H}$, the uniqueness analysis for this numerical solution $S_w$ is straightforward. Finally, due to the constraint $S_{n,h}^{k+1}=1-S_{w,h}^{k+1}$, we also have $S_{n,h}^{k+1} \in (0,1)$. This completes the proof of Theorem \ref{thmfirst}.
\end{proof}

\medskip
\begin{rem}
The above unique solvability and bound-preserving analysis for the case of $q_{\alpha} \ne 0,~ \alpha=w,n$,   remains valid by adding a bounded term $(-\mathcal{L}^{-1}_{\check{M}_{w,1}}(q_w^{k+1}/\phi_h) + \mathcal{L}^{-1}_{\check{M}_{n,1}}(q_n^{k+1}/\phi_h)  , S_w   )_h  $ in \eqref{first_convex}.
\end{rem}

\subsection{Unique solvability and bound preservation of the second-order scheme}
The second-order scheme \eqref{resecond_w}-\eqref{resecond_n} could be analyzed in a similar manner. An extension of this analytical framework to the two dimensional case is straightforward.

\medskip

\begin{theorem}\label{thmsec}
Given $S_{\alpha,h}^{k}, S_{\alpha,h}^{k-1} \in \mathcal{C}_h$, with $0< S_{\alpha,h}^{k}, S_{\alpha,h}^{k-1} <1, \alpha=w,n$, at a point-wise level. There exists a unique solution {\tiny } $S_{\alpha,h}^{k+1} \in \mathcal{C}_h$ to the second-order scheme \eqref{resecond_w}-\eqref{resecond_n}, with $0 <S_{\alpha,h}^{k+1} <1$, $\alpha=w,n$, under the condition that $\sigma_{wn}< \frac{1}{2}{ ({\sigma_{w}} +{\sigma_{n}})  }$.
\end{theorem}
\begin{proof}
Since $0< S_{w,h}^{k}, S_{w,h}^{k-1} <1$, and there are finite number of grid points, we assume that there exists $\delta_0$ such that
\begin{align*}
	\delta_0 \le S_{w,h}^{k}, S_{w,h}^{k-1} \le 1-\delta_0,
\end{align*}
for some $\delta_0$ that may depend upon $\Delta t$. Meanwhile, the numerical solution of \eqref{resecond_w} turns out to be a minimizer of the following discrete energy functional
\begin{align}
	\mathcal{J}^{k}(S_w)&:= \frac{1}{2 \Delta t}   \left \| S_w-S_{w,h}^{k} \right \|^2_{\mathcal{L}^{-1}_{\check{M}_{w,2}} }  +\frac{1}{2 \Delta t}   \left \| S_w-S_{w,h}^{k} \right \|^2_{\mathcal{L}^{-1}_{\check{M}_{n,2}} }  \\
	&+\Delta t \left( (S_w, {\rm ln} (S_w))_h +  ((1-S_w), {\rm ln}(1-S_w))_h \right)+ (S_w,f_h^k)_h \nonumber\\
	&+ \sigma_{w} (H^0_{S_{w,h}^k}(S_w) ,1)_h + \sigma_{n} (H^0_{1-S_{w,h}^k}(1-S_w) ,1)_h -  \frac{ \sigma_{wn} }{2}\left\|S_w\right\| _h^2 , \nonumber\\
	f_h^k&:= \left( -\sigma_{w}+\sigma_{n}+ \sigma_{wn}(1-S_{w,h}^k) + \Delta t (- {\rm ln} (S_{w,h}^k) + {\rm ln} (1-S_{w,h}^k)) \right),
\end{align}
over an admissible set $\Omega_{H}$. {\color{black} Recalling Lemma \ref{energypa}}, it is clear that $\mathcal{J}^{k}(S_w)$ is a strictly convex function over this domain if the coefficients satisfy $\sigma_{wn} < \frac{1}{2}{ ({\sigma_{w}} +{\sigma_{n}})  }$. 

{\color{black}
	Again, the minimization problem is transformed into an equivalent form
	\begin{align}
		\mathcal{F}^k(\varphi):&= \mathcal{J}^k(\varphi+\bar{S}_{w,h}^0)\nonumber\\
		&=\frac{1}{2 \Delta t}   \left \| \varphi+\bar{S}_{w,h}^0-S_{w,h}^{k} \right \|^2_{\mathcal{L}^{-1}_{\check{M}_{w,2}} }  +\frac{1}{2 \Delta t}   \left \| \varphi+\bar{S}_{w,h}^0-S_{w,h}^{k} \right \|^2_{\mathcal{L}^{-1}_{\check{M}_{n,2}} } \nonumber\\
		&+\Delta t  (\varphi+\bar{S}_{w,h}^0, {\rm ln} (\varphi+\bar{S}_{w,h}^0))_h  
		+\Delta t   ( (1-\varphi-\bar{S}_{w,h}^0), {\rm ln}(1-\varphi-\bar{S}_{w,h}^0))_h  \nonumber \\
		&+\sigma_{w} (H^0_{S_{w,h}^k}(\varphi+\bar{S}_{w,h}^0) ,1)_h  +  \sigma_{n} (H^0_{1-S_{w,h}^k}(1-\varphi-\bar{S}_{w,h}^0) ,1)_h \nonumber \\
		&-  \frac{\sigma_{wn} }{2}\left\| \varphi+\bar{S}_{w,h}^0 \right\|_h^2 + (\varphi+\bar{S}_{w,h}^0,f_h^k)_h,\nonumber
	\end{align}
	defined on the admissible set $\Omega^0_{H}$. A minimizer $\varphi \in \Omega^0_{H}$ is equivalent to the minimizer of $S_w=\varphi+\bar{S}_{w,h}^0 \in \Omega_{H}$. Next, we aim to prove that there exists a minimizer of $\mathcal{F}^k$ over the domain $\Omega_{H}^0$. Similar to case of the first-order scheme, we get $$ 1-\delta \ge \varphi^{\star}_{\vec{\alpha}_1} + \bar{S}_{w,h}^0 \ge  \bar{S}_{w,h}^0.$$
	The directional derivative becomes  
	\begin{align}
		d_s\mathcal{F}^k(\varphi^{\star}+s\psi)|_{s=0}&= \frac{1}{ \Delta t}  (\mathcal{L}^{-1}_{\check{M}_{w,2}}( \varphi^{\star}+\bar{S}_{w,h}^0 -S_{w,h}^{k}) , \psi)_h - \sigma_{wn} ((\varphi^{\star}+\bar{S}_{w,h}^0), \psi)_h  \nonumber\\
		&+  \frac{1}{ \Delta t}  (\mathcal{L}^{-1}_{\check{M}_{n,2}}( \varphi^{\star}+\bar{S}_{w,h}^0 -S_{w,h}^{k} ) , \psi)_h +  (f_h^k, \psi)_h  \nonumber\\
		& + \sigma_{w} (H^1_{S_{w,h}^k}(\varphi^{\star}+\bar{S}_{w,h}^0) ,\psi)_h - \sigma_{n} (H^1_{1-S_{w,h}^k}(1-\varphi^{\star}-\bar{S}_{w,h}^0) ,\psi)_h \nonumber\\
		& + \Delta t  ({\rm ln}(\varphi^{\star}+\bar{S}_{w,h}^0) - {\rm ln} (1-\varphi^{\star}-\bar{S}_{w,h}^0),\psi)_h .
	\end{align}
	Meanwhile, the direction is chosen as $\psi_{i,j,k}=\delta_{i,i_0}\delta_{j,j_0} \delta_{k,k_0}-  \delta_{i,i_1}\delta_{j,j_1} \delta_{k,k_1}$, and the derivative may be expressed as
	\begin{align}\label{dsf}
		\frac{1}{h^3}d_s\mathcal{F}^k(\varphi^{\star}+s\psi)|_{s=0}&= \frac{1}{ \Delta t}   \mathcal{L}^{-1}_{\check{M}_{w,2}}( \varphi^{\star}+\bar{S}_{w,h}^0 -S_{w,h}^{k})_{\vec{\alpha}_0 } +\Delta t {\rm ln}(\varphi_{\vec{\alpha}_0}^{\star}+\bar{S}_{w,h}^0)    \nonumber\\
		&- \frac{1}{ \Delta t}   \mathcal{L}^{-1}_{\check{M}_{w,2}}( \varphi^{\star}+\bar{S}_{w,h}^0 -S_{w,h}^{k})_{\vec{\alpha}_1} - \Delta t {\rm ln} (1-\varphi_{\vec{\alpha}_0}^{\star}-\bar{S}_{w,h}^0)  \nonumber \\
		& + \frac{1}{ \Delta t}   \mathcal{L}^{-1}_{\check{M}_{n,2}}( \varphi^{\star}+\bar{S}_{w,h}^0 -S_{w,h}^{k})_{\vec{\alpha}_0} - \Delta t {\rm  ln}(\varphi_{\vec{\alpha}_1}^{\star}+\bar{S}_{w,h}^0)  \nonumber \\ 
		&- \frac{1}{ \Delta t}   \mathcal{L}^{-1}_{\check{M}_{n,2}}( \varphi^{\star}+\bar{S}_{w,h}^0 -S_{w,h}^{k})_{\vec{\alpha}_1}+ \Delta t {\rm ln} (1-\varphi_{\vec{\alpha}_1}^{\star}-\bar{S}_{w,h}^0) \nonumber\\   
		&+ \sigma_{w} (H^1_{S_{w,h}^k}(\varphi_{\vec{\alpha}_0}^{\star}+\bar{S}_{w,h}^0) - H^1_{S_{w,h}^k}(\varphi_{\vec{\alpha}_1}^{\star}+\bar{S}_{w,h}^0) )\nonumber\\
		&- \sigma_{n} (H^1_{1-S_{w,h}^k}(1-\varphi_{\vec{\alpha}_0}^{\star}-\bar{S}_{w,h}^0) -  H^1_{1-S_{w,h}^k}(1-\varphi_{\vec{\alpha}_1}^{\star}-\bar{S}_{w,h}^0)  )\nonumber\\
		& - \sigma_{wn} ((\varphi^{\star}+\bar{S}_{w,h}^0)_{\vec{\alpha}_0}- (\varphi^{\star}+\bar{S}_{w,h}^0)_{\vec{\alpha}_1}) +  (f^k_{h,\vec{\alpha}_0}- f^k_{h,\vec{\alpha}_1} ).
	\end{align}
	Denote $S_w^{\star}:=\varphi^{\star}+\bar{S}_{w,h}^0$. The fact that $S_{w,\vec{\alpha}_0}^{\star}=\delta$, $ \bar{S}_{w,h}^0 \le S^{\star}_{w,\vec{\alpha}_1} \le 1-\delta$ indicates 
	\begin{align} \label{eqwm2}
		-\sigma_{wn} ( S_{w,\vec{\alpha}_0 }^{\star}-   S_{w,\vec{\alpha}_1 }^{\star}) \le 	\sigma_{wn} (1-2\delta), 
	\end{align}
	and 
	\begin{align}
		{\rm ln}( S^{\star}_{w,\vec{\alpha}_0}) - {\rm ln}( S^{\star}_{w,\vec{\alpha}_1} )- {\rm ln}(1-S^{\star}_{w,\vec{\alpha}_0}  ) +{\rm ln}(1-S^{\star}_{w,\vec{\alpha}_1}  ) \le {\rm ln} \frac{\delta}{ 1-\delta } - {\rm ln} \frac{\bar{S}_{w,h}^0}{ 1-\bar{S}_{w,h}^0 }.\nonumber
	\end{align} 
	Again, an application of Lemma \ref{lemL} implies that 
	\begin{align}
		-4C_2 \le B_w(\omega_{\vec{\alpha}_0}) - B_w(\omega_{\vec{\alpha}_1}) 
		+ B_n(\omega_{\vec{\alpha}_0}) - B_n(\omega_{\vec{\alpha}_1}) \le 4C_2,
	\end{align}
	where $B_w(\omega_{\vec{\alpha}}) = \mathcal{L}^{-1}_{\check{M}_{w,2}}(\omega_{\vec{\alpha}})$, $ B_n(\omega_{\vec{\alpha}}) = \mathcal{L}^{-1}_{\check{M}_{n,2}}(\omega_{\vec{\alpha}})$, $\omega_{\vec{\alpha}} = (S_w^{\star} - S_{w,h}^k)_{\vec{\alpha}}$.
	
	Recalling Lemma \ref{lemG}, $H_{a}^{1}(x)$ is an increasing function of $x$ for a fixed $a >0$. Based on this property, we are able to derive the following inequalities 
	\begin{align}
		&\sigma_{w} (H^1_{S_{w,h}^k}( S_{w,\vec{\alpha}_0}^{\star} )  ) = \sigma_{w} (H^1_{S_{w,h}^k}(\delta)) \le  \sigma_{w} (H^1_{S_{w,h}^k}(S_{w,h}^k)) = \sigma_{w} ({\rm ln}(S_{w,h}^k) +1), \\
		&\sigma_{w} (H^1_{S_{w,h}^k}(S_{w,\vec{\alpha}_1}^{\star}) ) \ge  \sigma_{w} H^1_{S_{w,h}^k} (\bar{S}_{w,h}^0 ),\\
		& \sigma_{n} (H^1_{1-S_{w,h}^k}(1-S_{w,\vec{\alpha}_0}^{\star})) =  \sigma_{n} (H^1_{1-S_{w,h}^k}(1-\delta))\ge \sigma_{n} (H^1_{1-S_{w,h}^k}(\frac{1}{4})) ,\\ 
		&\sigma_{n} (H^1_{1-S_{w,h}^k}(1-S_{w,\vec{\alpha}_1}^{\star})) \le \sigma_{n} (H^1_{1-S_{w,h}^k}(1-\bar{S}_{w,h}^0 )).
	\end{align}
	A careful simplification reveals that 
	\begin{align}
		&\quad \sigma_{w} (H^1_{S_{w,h}^k}(S_{w,\vec{\alpha}_0}^{\star}) - H^1_{S_{w,h}^k}(S_{w,\vec{\alpha}_1}^{\star}) ) \nonumber\\
		&\le  \sigma_{w} ({\rm ln}(S_{w,h}^k) +1-H^1_{S_{w,h}^k} (\bar{S}_{w,h}^0 ) ,\\
		&\quad  - \sigma_{n} (H^1_{1-S_{w,h}^k}(1-S_{w,\vec{\alpha}_0}^{\star}) -  H^1_{1-S_{w,h}^k}(1-S_{w,\vec{\alpha}_1}^{\star} )  ) \nonumber\\
		&\le    \sigma_{n} (-H^1_{1-S_{w,h}^k}(\frac{1}{4}) + H^1_{1-S_{w,h}^k}(1-\bar{S}_{w,h}^0 )   ),
	\end{align}
	and
	\begin{align}
		&\left  | f_h^k \right | = \left | -\sigma_{w}+\sigma_{n}+ \sigma_{wn}(1-S_{w,h}^k) + \Delta t (- {\rm ln} (S_{w,h}^k) + {\rm ln} (1-S_{w,h}^k))  \right | \nonumber\\
		&\quad\quad \le \sigma_{w}+\sigma_{n} + 2\sigma_{wn} + 2\Delta t \left| {\rm ln} \delta_0 \right |.
	\end{align} 
	This in turn gives 
	\begin{align}\label{eqf}
		-2( \sigma_{w}+\sigma_{n} + 2\sigma_{wn} + 2\Delta t \left| {\rm ln} \delta_0 \right | ) &\le 	(f^k_{h,\vec{\alpha}_0}- f^k_{h,\vec{\alpha}_1} )  \nonumber\\
		&\le  2( \sigma_{w}+\sigma_{n} + 2\sigma_{wn} + 2\Delta t \left | {\rm ln} \delta_0 \right | ).
	\end{align}
	As a consequence, a substitution of \eqref{eqwm2}-\eqref{eqf} into \eqref{dsf} yields the following directional derivative
	\begin{align}
		\frac{1}{h^3}d_s\mathcal{F}^k(\varphi^{\star}+s\psi)|_{s=0} &\le \Delta t( {\rm ln} \frac{\delta}{ 1-\delta } - {\rm ln} \frac{\bar{S}_{w,h}^0}{ 1-\bar{S}_{w,h}^0 } ) +2 ( \sigma_{w}+\sigma_{n} + 2\sigma_{wn} + 2\Delta t \left | {\rm ln} \delta_0 \right | ) \nonumber\\
		&+\sigma_{w} ({\rm ln}(S_{w,h}^k) +1-H^1_{S_{w,h}^k} (\bar{S}_{w,h}^0 ) ) +\sigma_{wn} (1-2\delta) +  4C_2 \Delta t^{-1}  \nonumber\\
		& + \sigma_{n} (-H^1_{1-{S_{w,h}^k}}(\frac{1}{4}) + H^1_{1-S_{w,h}^k}(1-\bar{S}_{w,h}^0 )   ) .
	\end{align}
	Denote
	\begin{align}
		C_4&= \sigma_{w} ({\rm ln}(S_{w,h}^k) +1-H^1_{S_{w,h}^k} (\bar{S}_{w,h}^0 ) ) 
		+  \sigma_{n} (-H^1_{1-S_{w,h}^k}(\frac{1}{4}) + H^1_{1-S_{w,h}^k}(1-\bar{S}_{w,h}^0 )   )\nonumber \\
		&+2( \sigma_{w}+\sigma_{n} + 2\sigma_{wn} + 2\Delta t \left |{ \rm ln} \delta_0 \right | ) +  4C_2 \Delta t^{-1} ,\nonumber
	\end{align}
	a constant that will, in general, depend on $\Delta t$, $\delta_0$ and $ \bar{S}_{w,h}^0 $. We may choose $\delta$ sufficiently small so that
	\begin{align}\label{delta_re1}
		\frac{1}{h^3}d_s\mathcal{F}^k(\varphi^{\star}+s\psi)|_{s=0} \le \Delta t( {\rm ln} \frac{\delta}{ 1-\delta } - {\rm ln} \frac{\bar{S}_{w,h}^0}{ 1-\bar{S}_{w,h}^0 } )+\sigma_{wn} (1-2\delta) +C_4 <0.
	\end{align}   
	In turn, the following inequality becomes valid, provided that $\delta$ satisfies \eqref{delta_re1}: 
	\begin{align} 
		\frac{1}{h^3}d_s\mathcal{F}^k(\varphi^{\star}+s\psi)|_{s=0} < 0.
	\end{align}
	This fact contradicts the assumption that $\mathcal{F}^k$ reaches its minimum at $\varphi^{\star}$, since the directional derivative is negative. Using similar arguments, the global minimum of $\mathcal{F}^k$ on $\Omega^0_{H,\delta}$ cannot occur at the boundary point $\varphi^{\star}$ such that $\varphi^{\star}_{\vec{\alpha}_0} + \bar{S}_{w,h}^0 = 1-\delta$ for some $\vec{\alpha}_0$, either. A combination of these two facts reveals that the global minimum of $\mathcal{F}^k$ over $\Omega^0_{H,\delta}$ can only possibly occur at an interior point, for sufficiently small $\delta>0$. Then we conclude that there must exist a solution $S_w=\varphi+\bar{S}_{w,h}^0  \in \Omega_{H}$ that minimizes $\mathcal{J}^k$. The existence of the numerical solution is confirmed. Furthermore, since $\mathcal{J}^k$ is a strictly convex function over $\Omega_{H}$, the uniqueness of this numerical solution $S_w$ follows from a straightforward monotonicity argument. Finally, due to the constraint $S_{n,h}^{k+1}=1-S_{w,h}^{k+1}$, it is clear that $S_{n,h}^{k+1} \in (0,1)$. The proof of Theorem \ref{thmsec} is complete.}
	\end{proof}
	
	\section{Mass conservation and energy stability}\label{section4}
	In this section, we rigorously prove that the proposed schemes satisfy both the original energy dissipation law and local mass conservation. 
	
	{\color{black} By \eqref{first_scheme1}-\eqref{first_scheme2} and \eqref{sec_scheme1}-\eqref{sec_scheme2}, it is obvious that the first- and second-order schemes ensure the local mass conservation law.}
	
	\begin{theorem}
Both the fully discrete first-order scheme \eqref{refirst_w}-\eqref{refirst_n} and second-order scheme \eqref{resecond_w}-\eqref{resecond_n} are locally mass-conservative with respect to the wetting phase $S_w$ and non-wetting phase $S_n$.

\end{theorem}

{\color{black} Next we define the discrete total energy as 	
\begin{equation} 
	\begin{aligned} 	
		& 
		\mathcal{E}_{tol,h}^{k}= ( \phi_h , F(S_{w,h}^k,S_{n,h}^k))_h, \quad \mbox{with} 
		\\
		& 
		F(S_{w,h}^k,S_{n,h}^k)=	\sum_{\alpha=w,n } \sigma _{\alpha }S^{k}_{\alpha,h }( {\rm ln} (S^{k}_{\alpha,h })-1)+   \sigma _{wn}S^{k}_{w,h} S^{k}_{n,h}.
	\end{aligned}
	\end{equation}}	
	The following original energy dissipation law is valid for the first-order scheme.
	
	\begin{theorem} \label{dis_energy dissipation}
Let $q_{\alpha}=0$ and the no-flow boundary condition $ \textbf{u}_{\alpha,1} \cdot \textbf{n} =0 $ are imposed, where  $\textbf{u}_{\alpha,1}=-\check{M}_{\alpha,1}  \nabla_h(p_h + \mu_{\alpha,h})$ and $\textbf{n} $ is the normal unit outward vector to $\partial \Omega$, we have the following discrete original energy dissipation law:
{\color{black}
	\begin{equation}\label{e_dis_energy1_noflow}
		\frac{ \mathcal{E}_{tol,h}^{k+1}  - \mathcal{E}_{tol,h}^{k} }{ \Delta t} \le -\sum_{\alpha =w,n} \Big\| (\frac{ ( S_{\alpha,h}^k )^m } {\eta_{\alpha}}  {K})^{\frac{1}{2}}   \nabla_h (p_h^{k+1}+\mu^{k+1}_{\alpha,h }) \Big\|_h^{2} \le 0 .
	\end{equation}
}
\end{theorem}
\begin{proof}
Because of the convexity of $S_{\alpha} {\rm ln}(S_{\alpha})$, $\alpha=w,n$, we see that
\begin{align}
	&S^{k+1}_{\alpha,h} {\rm ln}S_{\alpha,h}^{k+1} - S^{k}_{\alpha,h} {\rm ln}S_{\alpha,h}^{k}   \le  (S_{\alpha,h}^{k+1} - S_{\alpha,h}^k) ({\rm ln} S_{\alpha,h}^{k+1} + 1),\label{con_eq}\\
	&S^{k+1}_{w,h}S^{k+1}_{n,h}- S^{k}_{w,h}S^{k}_{n,h} \le  S_{n,h}^k (S_{w,h}^{k+1} - S_{w,h}^k) + S_{w,h}^{k}(S_{n,h}^{k+1} - S_{n,h}^k).\label{con_eq2}
\end{align}
In turn, the following inequality becomes available:  
\begin{align} \label{dis_binary fluids}
	F(S^{k+1}_{w,h},S^{k+1}_{n,h})- F(S^{k}_{w,h},S^{k}_{n,h}) &= \sum_{\alpha=w,n } \sigma _{\alpha }S^{k+1}_{\alpha,h }( {\rm ln} (S^{k+1}_{\alpha,h })-1)+   \sigma _{wn}S^{k+1}_{w,h} S^{k+1}_{n,h} \nonumber \\  
	&- \left(  \sum_{\alpha=w,n } \sigma _{\alpha }S^{k}_{\alpha,h }( {\rm ln} (S^{k}_{\alpha,h })-1)+   \sigma _{wn}S^{k}_{w,h} S^{k}_{n,h}  \right)  \\
	& \le 	\mu_{w,h}^{k+1} (S_{w,h}^{k+1} - S_{w,h}^k) + \mu_{n,h}^{k+1} (S_{n,h}^{k+1} - S_{n,h}^k). \nonumber
\end{align}
As a consequence, a combination of \eqref{dis_binary fluids},  \eqref{first_scheme1}, \eqref{first_scheme2} leads to 
\begin{equation*}
	\aligned
	\frac{  \mathcal{E}^{k+1}_{tol,h} - \mathcal{E}^{k}_{tol,h}   }{\Delta t} = & \Big(\phi_h, \frac{ F(S^{k+1}_{w,h},S^{k+1}_{n,h})- F(S^{k}_{w,h},S^{k}_{n,h}) }{\Delta t} \Big)_h \\
	\le & \Big(\phi_h,  \mu^{k+1}_{w,h}\frac{ S_{w,h}^{k+1} - S_{w,h}^k }{\Delta t}+ \mu^{k+1}_{n,h} \frac{ S_{n,h}^{k+1} - S_{n,h}^k  }{\Delta t} \Big)_h  \\
	= & \Big( \phi_h,  (p_h^{k+1}+\mu^{k+1}_{w,h})\frac{S_{w,h}^{k+1} - S_{w,h}^k }{\Delta t}+ (p_h^{k+1}+\mu^{k+1}_{n,h})\frac{S_{n,h}^{k+1} - S_{n,h}^k }{\Delta t} \Big)_h\\
	=&  -\sum_{\alpha =w,n} \Big\| (\frac{ ( S_{\alpha,h}^k )^m } {\eta_{\alpha}}  {K})^{\frac{1}{2}}   \nabla_h (p_h^{k+1}+\mu^{k+1}_{\alpha,h }) \Big\|_h^{2}  \le 0 , 
	\endaligned
\end{equation*}
in which the summation by parts formulas have been applied. The desired result \eqref{e_dis_energy1_noflow} has been proved.
\end{proof}

A similar estimate also holds true for the second-order scheme. 

\begin{theorem}\label{thm_energy2}
Let $q_{\alpha}=0$ and the no-flow boundary condition $ \textbf{u}_{\alpha,2} \cdot \textbf{n} =0 $ are imposed, where  $\textbf{u}_{\alpha,2}=-\check{M}_{\alpha,2}  \nabla_h(p_h + \mu_{\alpha,h})$ and $ \textbf{n} $ is the normal unit outward vector to $\partial \Omega$. The following discrete original energy dissipation law is valid: 
\begin{equation}\label{e_dis_energy2_noflow}
	\frac{ \mathcal{E}_{tol,h}^{k+1}  - \mathcal{E}_{tol,h}^{k} }{ \Delta t} \le  -\sum_{\alpha =w,n} \Big\|  (\frac{( \tilde{S}_{\alpha,h}^{k+1} )^m}{\eta_{\alpha}} {K})^{\frac{1}{2}} \nabla_h (p_h^{k+\frac{1}{2}}+\mu^{k+\frac{1}{2}}_{\alpha,h }) \Big\|_h ^{2} \le 0 .
\end{equation}
\end{theorem}

\begin{proof}
Taking \color{black} a discrete inner product with \eqref{resecond_muw} by $S_{w,h}^{k+1} - S_{w,h}^k$, and with \eqref{resecond_mun} by $ S_{n,h}^{k+1} - S_{n,h}^k $ respectively,  we get
\begin{align}\label{swmu2}
	(	S_{w}^{k+1} - S_{w}^k , \mu_{w,h}^{k+\frac{1}{2}})_h = ( \sigma_{w} ( { S_{w,h}^{k+1}{\rm ln} S_{w,h}^{k+1}- S_{w,h}^{k}{\rm ln} S_{w,h}^{k}}  - (S_{w,h}^{k+1} - S_{w,h}^k) )\\
	\quad\quad\quad\quad \quad\quad+ \sigma_{wn} S_{n,h}^{k+\frac{1}{2}}(S_{w,h}^{k+1} - S_{w_h}^k) + \Delta t ({\rm ln} S_{w,h}^{k+1} -{ \rm ln} S_{w,h}^{k})(S_{w,h}^{k+1} - S_{w,h}^k),1)_h ,\nonumber \\
	(S_{n,h}^{k+1} - S_{n,h}^k , \mu_{n,h}^{k+\frac{1}{2}})_h
	= (\sigma_{n} ( { S_{n,h}^{k+1}{\rm ln} S_{n,h}^{k+1}- S_n^{k}{\rm ln} S_{n,h}^{k}} -(S_{n,h}^{k+1} - S_{n,h}^k) ) \\
	\quad\quad \quad\quad\quad\quad+ \sigma_{wn} S_{w,h}^{k+\frac{1}{2}}(S_{n,h}^{k+1} - S_{n,h}^k) + \Delta t ({\rm ln} S_{n,h}^{k+1} - {\rm ln} S_{n,h}^{k})(S_{n,h}^{k+1} - S_{n,h}^k) , 1)_h.\nonumber
\end{align}
Meanwhile, the following identity is observed: 
\begin{align}\label{swsn_eq2}
	S^{k+1}_{w,h}S^{k+1}_{n,h}- S^{k}_{w,h}S^{k}_{n,h} =  S_{n,h}^{k+\frac{1}{2}} (S_{w,h}^{k+1} - S_{w,h}^k) + S_{w,h}^{k+\frac{1}{2}}(S_{n,h}^{k+1} - S_{n,h}^k).
\end{align}
In turn, a combination of \eqref{swmu2}-\eqref{swsn_eq2} results in
\begin{align*} 
	& F(S^{k+1}_{w,h},S^{k+1}_{n,h})- F(S^{k}_{w,h},S^{k}_{n,h}) = \sum_{\alpha=w,n } \sigma _{\alpha }S^{k+1}_{\alpha,h }( {\rm ln} (S^{k+1}_{\alpha,h })-1)+   \sigma _{wn}S^{k+1}_{w,h} S^{k+1}_{n,h}  \nonumber\\  
	&- \Big(  \sum_{\alpha=w,n } \sigma _{\alpha }S^{k}_{\alpha,h }( {\rm ln} (S^{k}_{\alpha,h })-1)+   \sigma _{wn}S^{k}_{w,h} S^{k}_{n,h}  \Big)  \nonumber\\
	& = 	\mu_{w,h}^{k+\frac{1}{2}} (S_{w,h}^{k+1} - S_{w,h}^k) + \mu_{n,h}^{k+\frac{1}{2}} (S_{n,h}^{k+1} - S_{n,h}^k) \nonumber\\
	&- \Delta t \left( ({\rm ln} S_{w,h}^{k+1} - {\rm ln} S_{w,h}^{k})(S_{w,h}^{k+1} - S_{w,h}^k) + ({\rm ln} S_{n,h}^{k+1} - {\rm ln} S_{n,h}^{k})(S_{n,h}^{k+1} - S_{n,h}^k) \right).
\end{align*}
Making use of the monotonicity of ${\rm ln}(x)-{\rm ln}(1-x)$ and the saturation constraint $S_{n,h}^{k+1} = 1 - S_{w,h}^{k+1}$, we see that 
\begin{align*}
	& ({\rm ln} S_{w,h}^{k+1} - {\rm ln} S_w^{k})(S_{w,h}^{k+1} - S_{w,h}^k) + ({\rm ln }S_{n,h}^{k+1} - {\rm ln} S_{n,h}^{k})(S_{n,h}^{k+1} - S_{n,h}^k)\nonumber\\
	=&({\rm ln} S_{w,h}^{k+1} - {\rm ln} S_{w,h}^{k} - {\rm ln} (1-S_{w,h}^{k+1}) + {\rm ln} (1-S_{w,h}^{k}) ) (S_{w,h}^{k+1} - S_{w,h}^k) \ge 0.
\end{align*}
Subsequently, the following inequality is obtained:  
\begin{align*}
	F(S^{k+1}_{w,h},S^{k+1}_{n,h})- F(S^{k}_{w,h},S^{k}_{n,h}) \le 	\mu_{w,h}^{k+\frac{1}{2}} (S_{w,h}^{k+1} - S_{w,h}^k) + \mu_{n,h}^{k+\frac{1}{2}} (S_{n,h}^{k+1} - S_{n,h}^k) .
\end{align*}
Finally, a combination of \eqref{sec_scheme1}, \eqref{sec_scheme4} and summation by parts formulas yields the following energy estimate:
{\color{black}
	\begin{equation*}
		\aligned
		\frac{  \mathcal{E}^{k+1}_{tol,h} - \mathcal{E}^{k}_{tol,h}   }{\Delta t} = &  \Big(\phi_h, \frac{ F(S^{k+1}_{w,h},S^{k+1}_{n,h})- F(S^{k}_{w,h},S^{k}_{n,h}) }{\Delta t} \Big)_h \\
		\le &  \Big( \phi_h, \mu^{k+\frac{1}{2}}_{w,h}\frac{ S_{w,h}^{k+1} - S_{w,h}^k }{\Delta t}+ \mu^{k+\frac{1}{2}}_{n,h} \frac{ S_{n,h}^{k+1} - S_{n,h}^k  }{\Delta t}  \Big)_h \\
		= & \Big(\phi_h, (p_h^{k+\frac{1}{2}}+\mu^{k+\frac{1}{2}}_{w,h})\frac{S_{w,h}^{k+1} - S_{w,h}^k }{\Delta t}+ (p_h^{k+\frac{1}{2}}+\mu^{k+\frac{1}{2}}_{n,h})\frac{S_{n,h}^{k+1} - S_{n,h}^k }{\Delta t} \Big)_h \\
		=& -\sum_{\alpha =w,n} \Big\|  (\frac{( \tilde{S}_{\alpha,h}^{k+1} )^m}{\eta_{\alpha}} {K})^{\frac{1}{2}} \nabla_h (p_h^{k+\frac{1}{2}}+\mu^{k+\frac{1}{2}}_{\alpha,h }) \Big\|_h ^{2} \le 0 . 
		\endaligned
	\end{equation*}
}
This completes the proof of Theorem \ref{thm_energy2}.
\end{proof}

\section{Error estimate}\label{section5}
For simplicity of presentation, we consider the first-order scheme \eqref{refirst_w}-\eqref{refirst_n} with $\check{M}_{\alpha,1}=1$, $\alpha=w,n$ and $\phi_h=1$. The convergence analysis for the non-constant mobility case will be considered in future works. Let $(S_w, S_n, p)$ be the exact solution of the PDE system~\eqref{e_modela}-\eqref{e_modelc}. With sufficiently regular initial data, it is reasonable to assume that the exact solution has regularity of class $\mathcal{R}$, where
\begin{equation}
S_\alpha \in W^{2,\infty}(0,T; L^{2}(\Omega)) \cap L^{\infty}(0,T; W^{4,\infty}(\Omega)) , 
\quad  p \in L^{\infty}(0,T;W^{4,\infty}(\Omega)) . 
\label{assumption:regularity.1}
\end{equation}
Meanwhile, we introduce  $S_{\alpha, N} (\, \cdot \, ,t) := {\mathcal{P}}_N S_\alpha (\, \cdot \, ,t)$, $p_N (\, \cdot \, ,t) := {\mathcal{P}}_N p (\, \cdot \, ,t)$, $\alpha = w, n$, the (spatial) Fourier Cosine projection of the exact solution into ${\mathcal {B}}^K$, the space of Cosine polynomials (to satisfy the Neumann boundary condition) of degree up to and including $K$ (with $N=2K +1$).  The following projection approximation is standard: if $S_\alpha, p \in L^\infty(0,T;H^\ell (\Omega))$, for any $\ell\in\mathbb{N}$ with $0 \le k \le \ell$, 
\begin{equation}
\begin{aligned} 
	& 
	\| S_{\alpha, N} - S_\alpha \|_{L^\infty(0,T;H^k)} \| \le C h^{\ell-k} \| S_\alpha \|_{L^\infty(0,T;H^\ell)}, 
	\\  
	& 
	\| p_N - p \|_{L^\infty(0,T;H^k)} \le C h^{\ell-k} \| p \|_{L^\infty(0,T;H^\ell)} . 
\end{aligned} 
\label{projection-est-0}
\end{equation}
In particular, we observe that $S_{w, N} + S_{n, N} \equiv 1$, which comes the Fourier Cosine projection of the saturation constraint identity $S_w + S_n \equiv 1$. 

By $S_{\alpha, N}^m$, $p_N^m$ we denote $S_{\alpha, N} ( \, \cdot \,, t^m)$ and $p_N (\, \cdot \, , t^m)$, respectively, with $t^m = m \cdot \Delta t$. Since $S_{\alpha, N} \in {\mathcal{B}}^K$, the mass conservative property is available at the discrete level:
\begin{equation} 
\overline{S}_{\alpha, N}^m = \frac{1}{|\Omega|}\int_\Omega \,S_{\alpha, N} (\, \cdot \,, t^m) \, d {\bf x} = \frac{1}{|\Omega|}\int_\Omega \, S_{\alpha, N} ( \, \cdot \,, t^{m-1}) \, d {\bf x} = \overline{S}_{\alpha, N}^{m-1} ,  \, \, \, 
\alpha = w, n , \label{mass conserv-1-1} 
\end{equation}
for any $m \in \mathbb{N}$. On the other hand, the numerical solution of $S_\alpha$ in \eqref{first_scheme1}-\eqref{first_scheme5} is also mass conservative at the discrete level:
\begin{equation}
\overline{S}_{\alpha, h}^m = \overline{S}_{\alpha, h}^{m-1} ,  \, \, \,  \alpha = w, n  , 
\quad \forall \ m \in \mathbb{N} . 
\label{mass conserv-2}
\end{equation}
Of course, we could use the mass conservative projection for the initial data:  
\begin{equation}
(S_{\alpha, h}^0)_{i,j,k} := S_{\alpha, N} (p_i, p_j, p_k, t=0) , \quad \alpha = w, n . 
\label{initial data-0}
\end{equation}	
The error grid function is defined as
\begin{equation}
e_{\alpha}^k =S_{\alpha, N}^k -S_{\alpha,h}^k, \, \, \, e_p^k=p_N^k -p^k_h, \, \, \, 
\quad \forall k \in \mathbb{N} .
\label{error function-1}
\end{equation}
Therefore, it follows that  $\overline{e}_w^k = \overline{e}_n^k =0$, for any $k \in \mathbb{N}$,  so that the discrete norm $\| \, \cdot \, \|_{-1,h}$ is well defined for the error grid function.

The error estimate for the first-order scheme is stated as follows.
\medskip
\begin{theorem}\label{therror}
Set $ \gamma_0 := (\sqrt{\sigma_{w}}+\sqrt{\sigma_{n}})^2 - 2\sigma_{wn} > 0$. We make a regularity assumption~\eqref{assumption:regularity.1} for the exact solution. Provided that $\Delta t$ and $h$ are sufficiently small, the following convergence estimate is valid for the first-order scheme \eqref{refirst_w}-\eqref{refirst_n}: 
\begin{align}
	&\| e_\alpha^{m+1}\|_{-1,h}^{2} +\Delta t \sum_{k=0}^{m}  \| e_\alpha^{k+1} \|_h^2 
	\le   C(\Delta t^2+h^4 ), \quad \alpha = w , n , \label{error1} 
\end{align}
where \text{C} is a constant independent of $\Delta t$ and $h$.
\end{theorem}

\begin{proof} 
An application of the Taylor expansion to~\eqref{e_modela}-\eqref{e_modelb} (with $q_{\alpha}=0$) implies that  
\begin{equation}
	\begin{aligned} 
		&
		\frac{S_{w, N}^{k+1}-S_{w, N}^k}{\Delta t} =\Delta_h ( p_N^{k+1} + \sigma_{w} {\rm ln}(S_{w, N}^{k+1} )  + \sigma_{wn} (1-S_{w, N}^k ) ) + \tau_w^k ,  
		\\
		& 
		\frac{S_{n, N}^{k+1}-S_{n, N}^k}{\Delta t} =\Delta_h ( p_N^{k+1} + \sigma_{n} {\rm ln}(S_{n, N}^{k+1} )  + \sigma_{wn} S_{w, N}^k ) + \tau_n^k , 
	\end{aligned} 
	\label{consistency-1} 
\end{equation} 
with $\| \tau_w^k \|_{-1, h} , \, \| \tau_n^k \|_{-1, h} \le C (\Delta t + h^2)$. In turn, a difference between these two equations yields 
\begin{equation}
	\frac{2 ( S_{w, N}^{k+1}-S_{w, N}^k)}{\Delta t} =\Delta_h ( \sigma_{w} {\rm ln}(S_{w, N}^{k+1} )  - \sigma_{n} {\rm ln}(1-S_{w, N}^{k+1} )+ \sigma_{wn} (1- 2 S_{w, N}^k ) ) + \tau_0^k ,  \label{consistency-2} 
\end{equation} 
with $\tau_0^k = \tau_w^k - \tau_n^k$, $\| \tau_0^k \|_{-1, h} \le C (\Delta t + h^2)$. Notice that the saturation constraint identity, $S_{w, N} + S_{n, N} \equiv 1$, has been repeatedly applied in the derivation. Meanwhile, in terms of the numerical solution, a difference between \eqref{first_scheme1} and \eqref{first_scheme2} leads to a similar equation: 
\begin{equation}
	\frac{2 ( S_{w, h}^{k+1}-S_{w, h}^k)}{\Delta t} =\Delta_h ( \sigma_{w} {\rm ln}(S_{w, h}^{k+1} )  - \sigma_{n} {\rm ln}(1-S_{w, h}^{k+1} )+ \sigma_{wn} (1- 2 S_{w, h}^k ) ) ,  \label{consistency-3} 
\end{equation} 
with a repeated application of the discrete saturation constraint identity $S_{w, h} + S_{n, h} \equiv 1$. As a result, subtracting~\eqref{consistency-3} from \eqref{consistency-2} gives 
\begin{equation}\label{error equation-1}
	\frac{2 (e_w^{k+1}-e_w^{k})}{\Delta t}= \Delta_h (  \mathcal{G}(S_{w, N}^{k+1})-\mathcal{G}(S_{w,h}^{k+1})- 2\sigma_{wn} e_w^k )+ \tau_0^k ,
\end{equation}
where $\mathcal{G}(x)=\sigma_{w}{\rm ln}(x)-\sigma_{n}{\rm \ln}(1-x)$. 

Taking a discrete inner product with \eqref{error equation-1} by $(-\Delta_h)^{-1} e_w^{k+1}$ leads to 
\begin{equation} 
	\begin{aligned} 
		& 
		\frac{2}{\Delta t} \langle e_w^{k+1}-e_w^k, e_w^{k+1} \rangle_{-1, h} 
		+ ( \mathcal{G}(S_w^{k+1}) - \mathcal{G}(S_{w,h}^{k+1}) , e_w^{k+1} )_h 
		\\
		=  & 
		2 \sigma_{wn} ( e_w^k  , e_w^{k+1} )_h 
		+  \langle \tau_0^k , e_w^{k+1} \rangle_{-1, h} , 
	\end{aligned} 
	\label{convergence-1} 
\end{equation}
in which the summation-by-parts formulas have been repeatedly applied. The estimates for the temporal differentiation, the concave and the local truncation terms are straightforward: 
\begin{align} 
	& 
	2  \langle e_w^{k+1}-e_w^k, e_w^{k+1} \rangle_{-1, h}  
	= \| e_w^{k+1} \|_{-1, h}^2 - \| e_w^k \|_{-1, h}^2  + \| e_w^{k+1}-e_w^k \|_{-1, h}^2 , 
	\label{convergence-2} 
	\\
	& 
	2 \sigma_{wn} ( e_w^k  , e_w^{k+1} )_h  
	\le \sigma_{wn} ( \| e_w^k  \|_h^2 + \| e_w^{k+1} \|_h^2 )  ,  \label{convergence-3}  
	\\
	&
	\langle \tau_0^k , e_w^{k+1} \rangle_{-1, h} 
	\le \frac12 ( \| \tau_0^k \|_{-1, h}^2 + \| e_w^{k+1} \|_{-1, h}^2) . \label{convergence-4} 
\end{align} 
In terms of the nonlinear error estimate, an application of Lagrange mean value theorem implies that  
\begin{equation}
	\mathcal{G}(S_{w, N}^{k+1})-\mathcal{G}(S_{w,h}^{k+1})=\mathcal{G}'(\xi^{k+1} )e_w^{k+1} ,  
	\quad \mathcal{G}'(\xi^{k+1} ) = \frac{\sigma_{w}}{\xi^{k+1}}+\frac{\sigma_{n}}{1-\xi^{k+1}}  ,  
	\label{convergence-5-1} 
\end{equation}
in which $\xi^{k+1}$ is between $S_{w, N}^{k+1}$ and $S_{w,h}^{k+1}$, at a point-wise level. Meanwhile, by the bound-preserving property for both the exact and numerical solutions, we see that $0 < \xi^{k+1} < 1$. Moreover, a Calculus-style analysis reveals a lower bound for the nonlinear coefficient: 
\begin{align} 
	\mathcal{G}'(\xi^{k+1} ) &= \frac{\sigma_{w}}{\xi^{k+1}}+\frac{\sigma_{n}}{1-\xi^{k+1}}\nonumber\\
	&\ge C_{min}=\frac{\sigma_{w}(\sigma_{w}-\sigma_{n})}{\sigma_{w}-\sqrt{\sigma_{w}\sigma_{n}}} + \frac{\sigma_{n}(\sigma_{w}-\sigma_{n})}{-\sigma_{n}+\sqrt{\sigma_{w}\sigma_{n}}}=(\sqrt{\sigma_{w}}+\sqrt{\sigma_{n}})^2 . 
	\label{convergence-5-2} 
\end{align}
In turn, the following estimate becomes available: 
\begin{equation} 
	( \mathcal{G}(S_{w, N}^{k+1}) - \mathcal{G}(S_{w,h}^{k+1}) , e_w^{k+1} )_h  
	= (  \mathcal{G}'(\xi^{k+1} ) e_w^{k+1} ,  e_w^{k+1} )_h 
	\ge C_{min}  \|  e_w^{k+1} \|_h^2 .  \label{convergence-5-3} 
\end{equation}
Subsequently, a substitution of~\eqref{convergence-2}-\eqref{convergence-4} and \eqref{convergence-5-3} into \eqref{convergence-1} results in 
\begin{equation} 
	\begin{aligned} 
		& 
		\frac{1}{\Delta t} ( \| e_w^{k+1} \|_{-1, h}^2 - \| e_w^k \|_{-1, h}^2 ) 
		+ C_{min}  \|  e_w^{k+1} \|_h^2 
		\\
		\le & 
		\sigma_{wn} ( \| e_w^k  \|_h^2 + \| e_w^{k+1} \|_h^2 ) 
		+ \frac12 ( \| \tau_0^k \|_{-1, h}^2 + \| e_w^{k+1} \|_{-1, h}^2 ) . 
	\end{aligned} 
	\label{convergence-6} 
\end{equation}

A summation of~\eqref{convergence-6} over time leads to 
\begin{equation} 
	\begin{aligned} 
		& 
		\| e_w^{m+1} \|_{-1, h}^2 + \gamma_0 \Delta t \sum_{k=1}^{m+1} \|  e_w^k \|_h^2 
		\le \frac{\Delta t}{2} \sum_{k=1}^{m+1} \| e_w^k \|_{-1, h}^2 
		+ \frac{\Delta t}{2}  \sum_{k=0}^{m} \| \tau_0^k \|_{-1, h}^2 ,  
	\end{aligned} 
	\label{convergence-9} 
\end{equation} 
in which the identity $\gamma_0 = C_{min} - 2 \sigma_{wn}$ has been used. Finally, an application of discrete Gronwall inequality yields the desired convergence estimate~\eqref{error1}, for $\alpha =w$. The analysis for $\alpha =n$ is equivalent, since $e_n^k = - e_w^k$, for any $k \ge 0$. This finishes the proof of Theorem~\ref{therror}. 
\end{proof} 

\begin{rem} 
In the PDE system~\eqref{e_modela}-\eqref{e_modelc}, the pressure $p$ plays a key role of a Lagrange multiplier to enforce the saturation constraint $S_w + S_n \equiv 1$. In the numerical system~\eqref{first_scheme1}-\eqref{first_scheme5}, the pressure variable plays a similar role. Because of its Lagrange multiplier nature, both its regularity estimate in the PDE analysis and the numerical error estimate turn out to be a highly challenging issue. To overcome the difficulty associated with the error estimate for the Lagrange multiplier, we take the difference between the evolutionary equations for the two saturation variables, in both the consistency analysis~\eqref{consistency-2} (applied to the projection solution) and the numerical system~\eqref{consistency-3}. In turn, the pressure variable has been cancelled by taking a difference, and the subsequent error estimates turn out to be more straightforward. This approach is the key point in the convergence analysis. 
\end{rem}

\begin{rem} 
The nonlinear error estimates~\eqref{convergence-5-1}-\eqref{convergence-5-3} come from the convexity of the logarithmic terms in the free energy expansion. Such a convexity ensures a lower bound of the diffusion error inner product in \eqref{convergence-5-3}, based on a point-wise lower bound~\eqref{convergence-5-2}. However, if an $\ell^\infty(0, T; \ell^2) \cap \ell^2 (0, T; H_h^1)$ error estimate is performed, the point-wise lower bound~\eqref{convergence-5-2} is not sufficient to guarantee a lower bound for the associated gradient error estimate analogous of~\eqref{convergence-5-3}, since the gradient estimate for ${\mathcal{G}}' (\xi^{k+1})$ is highly challenging. In fact, to accomplish such a gradient estimate, a combination of a rough and a refined error estimates is needed to establish a phase separation property for the saturation variables, i.e., a uniform distance between the physical variables and the singular limit values of 0 and 1; see the related work of convergence analysis for the Poisson-Nernst-Planck (PNP) system~\cite{liu2021positivity}. For simplicity of presentation, we only provide the $\ell^\infty(0, T; H_h^{-1}) \cap \ell^2 (0, T; \ell^2)$ error estimate in this article, and the $\ell^\infty(0, T; \ell^2) \cap \ell^2 (0, T; H_h^1)$ error estimate will be left in the future works. 
\end{rem}

\begin{rem} 
Similarly, if a non-constant mobility function is considered, the convergence analysis becomes much more complicated, since a direct application of the $\ell^\infty(0, T; H_h^{-1}) \cap \ell^2 (0, T; \ell^2)$ error estimate fails. Instead, an $\ell^\infty(0, T; \ell^2) \cap \ell^2 (0, T; H_h^1)$ error estimate is necessary to theoretically establish the convergence result, and this error estimate has to be based on a combination of a rough and a refined error estimates, such as the one reported in \cite{liu2021positivity} for the PNP system. The convergence analysis for the case of a non-constant mobility function will also be left in the future works, for the sake of brevity of this article. 
\end{rem} 

\section{Numerical experiments}\label{section6}
In this section, we provide some numerical experiments to verify the accuracy and efficiency of the proposed first-order scheme \eqref{refirst_w}-\eqref{refirst_n} and second-order scheme \eqref{resecond_w}-\eqref{resecond_n}. In terms of the energy parameters, we take $\sigma_{w}=\frac{{\overline{\sigma}_{w}}}{\sqrt{K}},~\sigma_{n}=\frac{{\overline{\sigma}_{n}}}{\sqrt{K}},~\sigma_{wn}=\frac{{\overline{\sigma}_{wn}}}{\sqrt{K}}$. The convergence criterion for the Newton iteration is chosen as 
$$ \left \| S_{w,h}^{k+1,L+1}-S_{w,h}^{k+1,L}  \right \|_{\infty} < \epsilon,$$		
where $\epsilon=10^{-6}$ is the error tolerance and $L$ is the iteration count. Meanwhile, the maximum number of iterations is set to be 100 for all examples.

\subsection{ Example 1}\label{example1}  
In this example, we test the accuracy of the numerical scheme. The computational domain is chosen as $\Omega= (0, 1m)^2$. The porosity and permeability are uniformly distributed in space, with $\phi= 0.8$  and $\text{K}=0.1 m^2$. The energy parameters are set as $\overline{\sigma}_{w}=0.58 \ pa\cdot m $, $\overline{\sigma}_{n}=0.05 \ pa\cdot m$, $\overline{\sigma}_{wn}=0.36 \ pa\cdot m$. We take the viscosities as $\eta_{w}=1 \ pa\cdot s $, $\eta_{n}=0.5 \ pa \cdot s$, and the parameter $m=2$. The exact solutions for the wetting-phase saturation and the pressure are given by
$$ S_w=e^{-t}(\frac{1}{4}x^2 (1-x)^2 y^2 (1-y)^2+\frac{1}{2}),~~ p=e^{-t}(\frac{1}{2}cos(\pi x)cos(\pi y)). $$

\begin{table}[!t]  
\renewcommand{\arraystretch}{1.1}
\small
\centering
{\caption{The numerical errors and convergence rates for the first-order scheme at $T =1s$ with $\tau=h^2$.}	\label{table1_example1}}
\begin{tabular}{ccccc}\hline
	$\tau$    &$ {\color{black}\left \| e_{w}  \right \|_{\ell^{2}(0,T;\ell^2(\Omega))  }  } $    &Rate  &$\left \| e_{p} \right \|_{\ell^{2}(0,T;\ell^2(\Omega))  } $   &Rate  \\ \hline
	$1/16$      &0.0082      & ---    &0.0220     & --- \\
	$1/64$    &0.0021       &0.9965   &0.0055    &1.0039\\
	$1/256$   &5.1569e-04   &0.9989   &0.0014    &1.0008\\
	$1/1024$   &1.2897e-04   &0.9997   &3.4206e-04  &1.0001\\
	$1/4096$   &3.2242e-05  &1.0000   &8.5553e-05  &0.9997\\
	\hline
\end{tabular}
\end{table}
\begin{table}[!t]  
\renewcommand{\arraystretch}{1.1}
\small
\centering
{\caption{The numerical errors and convergence rates for the second-order scheme at $T =1s$ with $\tau=\frac{1}{2}h$.}\label{table2_example1}}
\begin{tabular}{ccccc}\hline
	$h$  &$\left \| e_{w} \right \|_{\ell^{2}(0,T;\ell^2(\Omega))  } $    &Rate  &$\left \| e_{p} \right \|_{\ell^{2}(0,T;\ell^2(\Omega))  } $    &Rate   \\ \hline
	$1/8$    &8.9136e-05    & ---     &0.0039     & --- \\
	$1/16$   &2.4753e-05   &1.8484   &9.6942e-04  &2.0179 \\
	$1/32$   &6.5328e-06    &1.9218  &2.4108e-04    &2.0076\\
	$1/64$   &1.6785e-06    &1.9606  &6.0065e-05    &2.0049 \\
	$1/90$   &8.5645e-07    &1.9736  &3.0315e-05    &2.0057 \\
	\hline
\end{tabular}
\end{table}
We calculate the error between the numerical and the exact solutions, and present the results at $T=1s$ in Tables \ref{table1_example1} and \ref{table2_example1}, respectively. It is observed that the proposed schemes achieve the first- and second-order convergence rates in time for both $S_w$ and $p$, which are consistent with the theoretical analysis in Theorem \ref{therror}.

Moreover, we use different analytical solutions to verify the convergence rates. The porosity is uniformly distributed in space with $\phi=0.9$, and the remaining parameters are consistent with the above ones. Here the initial conditions and the right-hand side of the equation are computed based on the analytic solutions provided below
$$ 	S_w=e^{-t}(\frac{1}{10}cos(\pi x) cos(\pi y)+\frac{1}{2}),~~p=e^{-t}(\frac{1}{2}cos(\pi x)cos(\pi y)).  $$
We present the numerical results for the first- and second-order schemes at the final time $T = 1s$, using time step sizes $\tau=h^2$ and $\tau=h$, respectively, in Tables \ref{table1_example2} and \ref{table2_example12}. It is clear that the results are consistent with the theoretical results in Theorem \ref{therror}.
\begin{table}[!t]  
\renewcommand{\arraystretch}{1.1}
\small
\centering
{		\caption{ The numerical errors and convergence rates for the first-order scheme at $T =1s$ with $\tau=h^2$.}	\label{table1_example2}}
\begin{tabular}{ccccc}\hline
	$\tau$    &$\left \| e_{w} \right \|_{\ell^{2}(0,T;\ell^2(\Omega))  } $    &Rate  &$\left \| e_{p} \right \|_{\ell^{2}(0,T;\ell^2(\Omega))  } $    &Rate  \\ \hline
	$1/16$      &0.0083     & ---     &0.0229     & --- \\
	$1/64$    &0.0021       &0.9948   &0.0057       &1.0053\\
	$1/256$   &5.2428e-04   &0.9985  &0.0014    &1.0012\\
	$1/1024$   &1.3114e-04   &0.9996   &3.5374e-04  &1.0002\\
	$1/4096$   &3.2785e-05  &1.0000   &8.8474e-05  &0.9997\\
	\hline
\end{tabular}
\end{table}
\begin{table}[!t]  
\renewcommand{\arraystretch}{1.1}
\small
\centering
\vspace{-3pt}
{	\caption{The numerical errors and convergence rates  for the second-order scheme  at $T =1s$ with $\tau=h$.}\label{table2_example12}}
\begin{tabular}{ccccc}\hline
	$h$  &$\left \| e_{w} \right \|_{\ell^{2}(0,T;\ell^2(\Omega))  } $    &Rate  &$\left \| e_{p} \right \|_{\ell^{2}(0,T;\ell^2(\Omega))  } $    &Rate   \\ \hline
	$1/8$    &8.6681e-04    & ---     &0.0036     & --- \\
	$1/16$   &2.3214e-04   &1.9007   &8.9288e-04  &2.0196 \\
	$1/32$   &6.0010e-05    &1.9517  &2.2104e-04    &2.0141\\
	$1/64$   &1.5258e-05    &1.9756  &5.4862e-05    &2.0104 \\
	$1/90$   &7.7540e-06    &1.9855  &2.7712e-05    &2.0033 \\
	\hline
\end{tabular}
\vspace{-3pt}
\end{table}
\subsection{Example 2}\label{example2} 
We initially focus on a benchmark problem from \cite{kou2022energy} and consider a closed region with $\Omega=[0, 100m]^2$. For the closed system, we use the no-flow boundary condition, i.e., $\textbf{u} \cdot \textbf{n}=0$ on the entire boundary $\partial \Omega$, where \textbf{n} denotes the outward unit normal vector to $\partial \Omega$. The initial distributions of porosity, permeability and wetting-phase saturation are illustrated in Figure \ref{exp2_initaial}. 
The viscosities are taken as $\eta_{w}=0.9\ cp$ and $\eta_{n}=0.1\ cp$. The energy parameters are given by $\sigma_{w}= 11.655 \ pa, \sigma_{n}= 1.0796 \ pa, \sigma_{wn}= 7.424 \ pa$ for the lower permeable region, and $\sigma_{w}= 5.8275 \ pa, \sigma_{n}= 0.5398 \ pa, \sigma_{wn}= 3.712 \ pa$ for the higher permeable region and the parameter $m=3$. A uniform $50 \times 50$ mesh is used, and the time step size is chosen as $\tau= 0.5\ day$.

\begin{figure}[!t]
\centering
\includegraphics[scale=0.30]{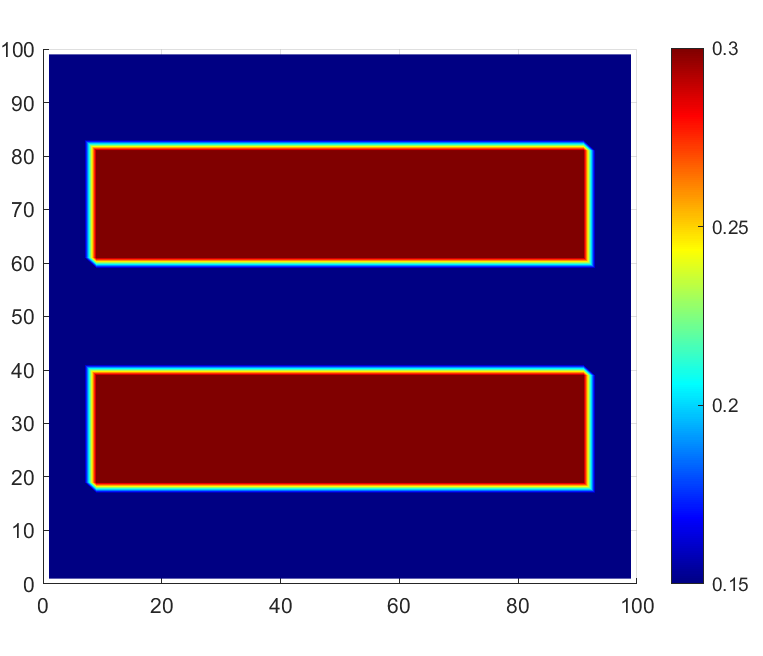}
\includegraphics[scale=0.30]{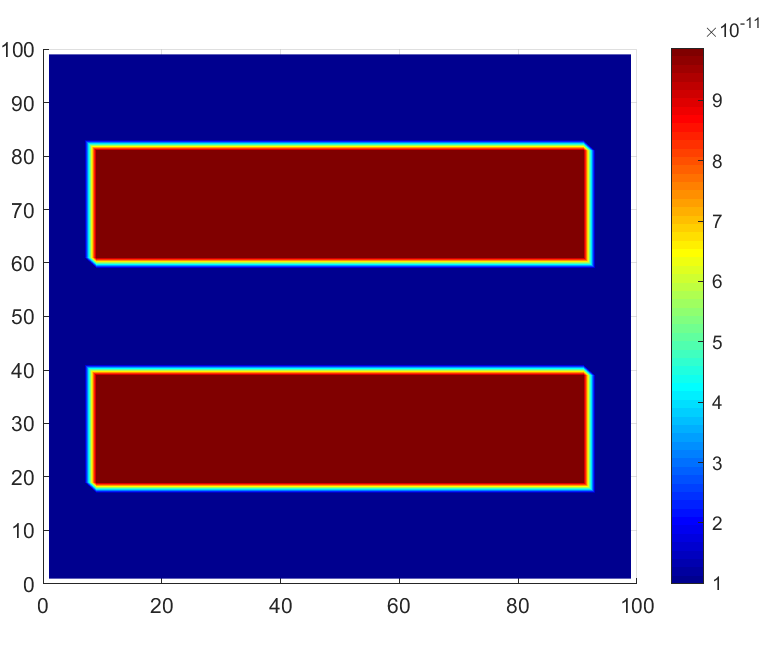}
\includegraphics[scale=0.30]{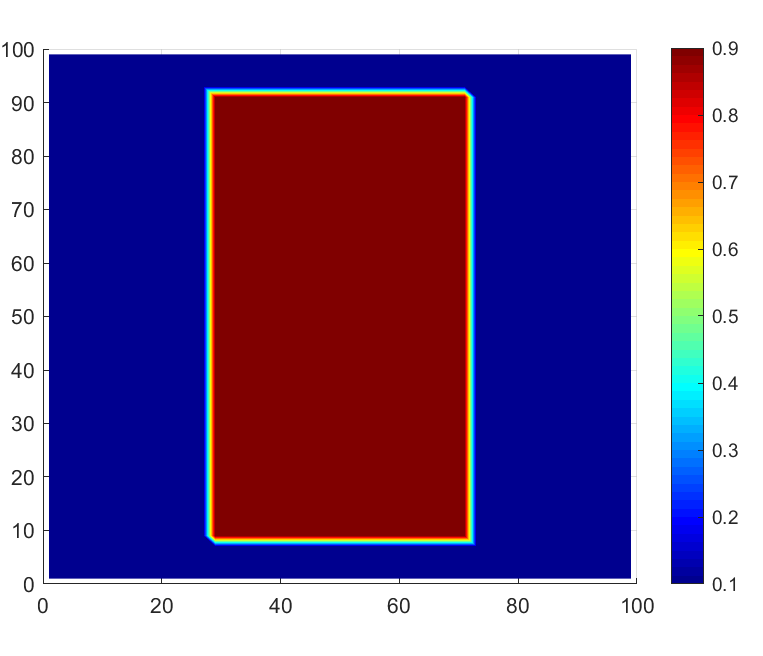}
\caption{The initial distributions of porosity(left), permeability(center) and wetting-phase saturation(right) {\color{black} in a regular domain}. }
\label{exp2_initaial}
\end{figure}

In Figure \ref{exp2_sw}, we plot the wetting-phase saturation contours at different time instants. It is clear that the wetting-phase fluid moves from areas with high permeability to the ones with low permeability. It is observed from Figure \ref{exp2_mu} that the distribution of chemical potential affects that of saturation, which indicates that chemical potential has a significant influence on the saturation distribution.  We also depict the minimum and maximum values for the wetting-phase saturation, original energy profiles and the evolution of total mass for both phases in Figure \ref{exp2_bound}, which are consistent with the theoretical analysis.


\begin{figure}[!t]
\centering
\includegraphics[scale=0.30]{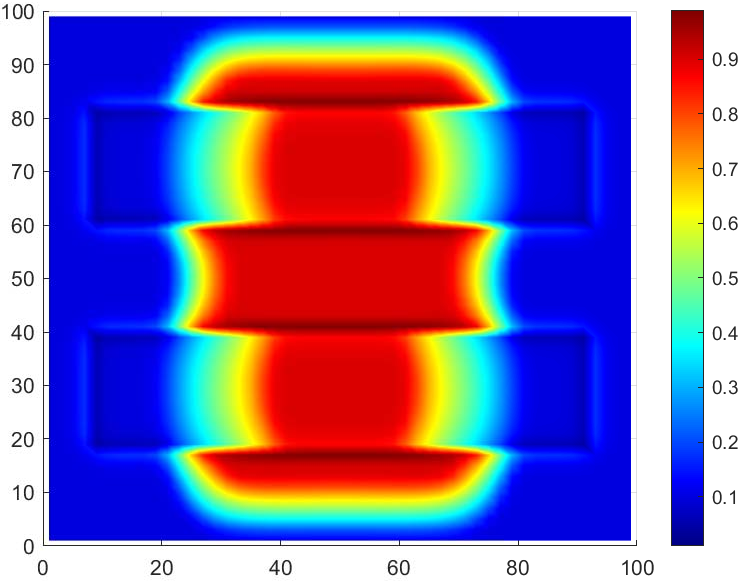} 
\includegraphics[scale=0.30]{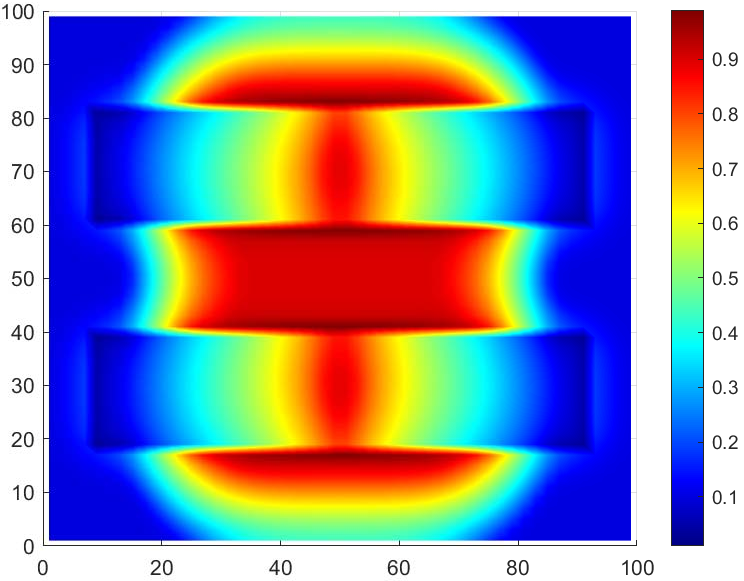}
\includegraphics[scale=0.30]{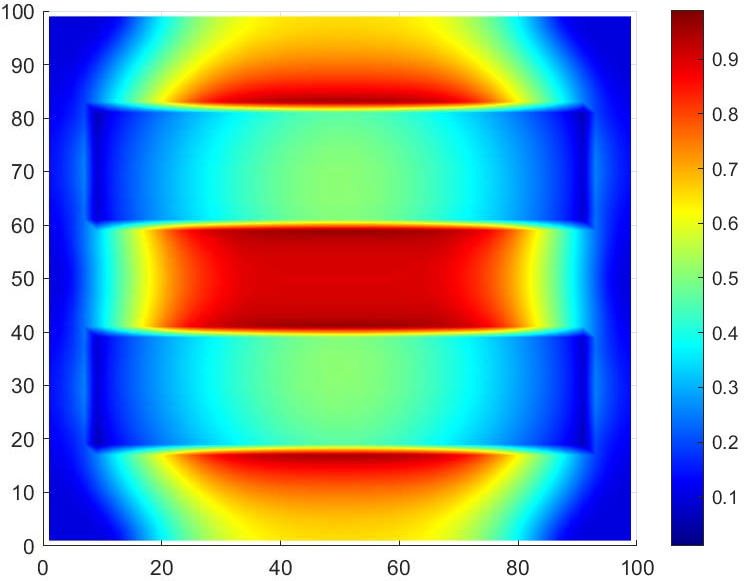}
\caption{The distributions of wetting-phase saturation at 4.11 year, 13.70 year, 34.25 year in 50 $\times$ 50 mesh {\color{black} in a regular domain}.}
\label{exp2_sw}
\end{figure}
\begin{figure}[!t]
\centering
\includegraphics[scale=0.30]{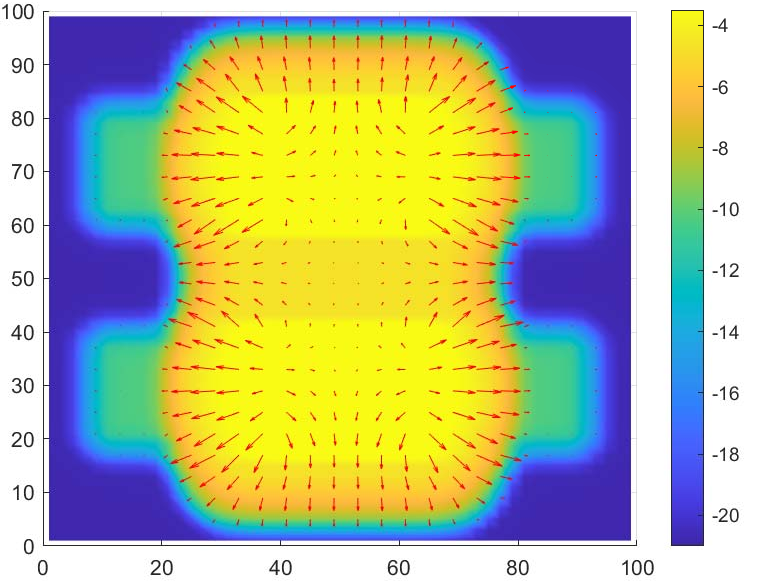}
\includegraphics[scale=0.30]{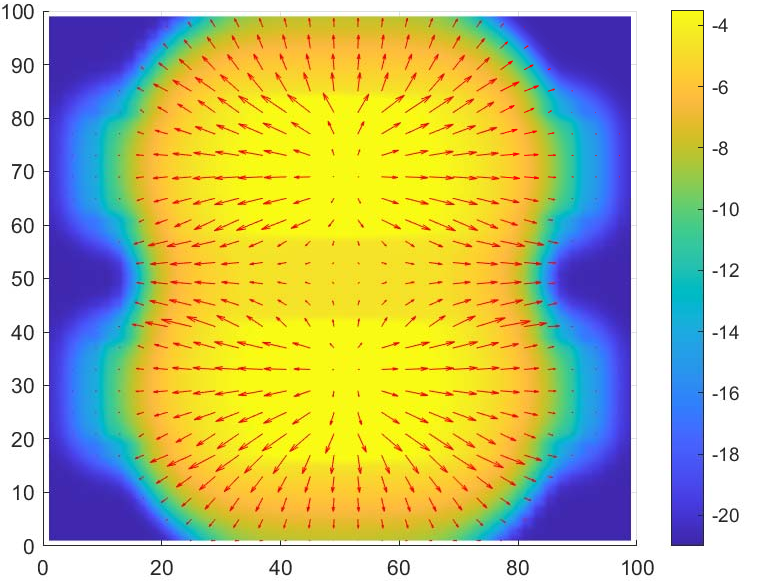}
\includegraphics[scale=0.30]{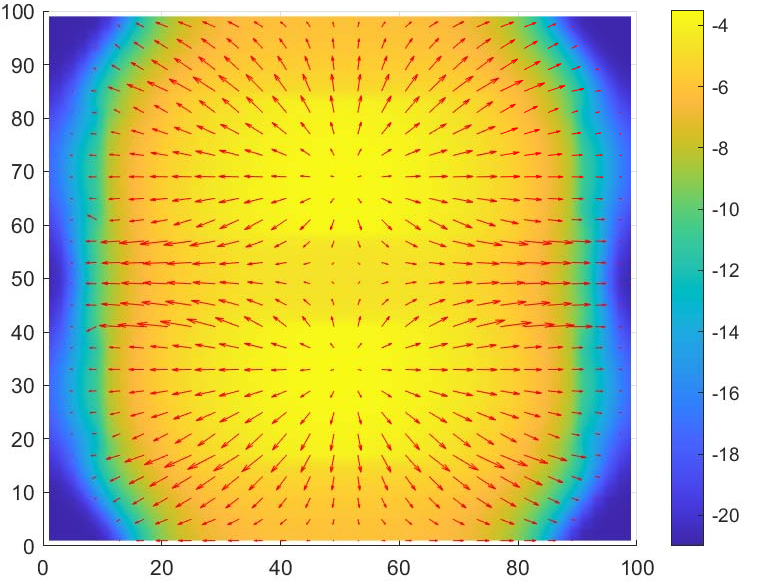}
\caption{The distributions of chemical potential at 4.11 year, 13.70 year, 34.25 year in 50 $\times$ 50 mesh {\color{black} in a regular domain}. }
\label{exp2_mu}
\end{figure}
\begin{figure}[htp]
\centering
\includegraphics[scale=0.30]{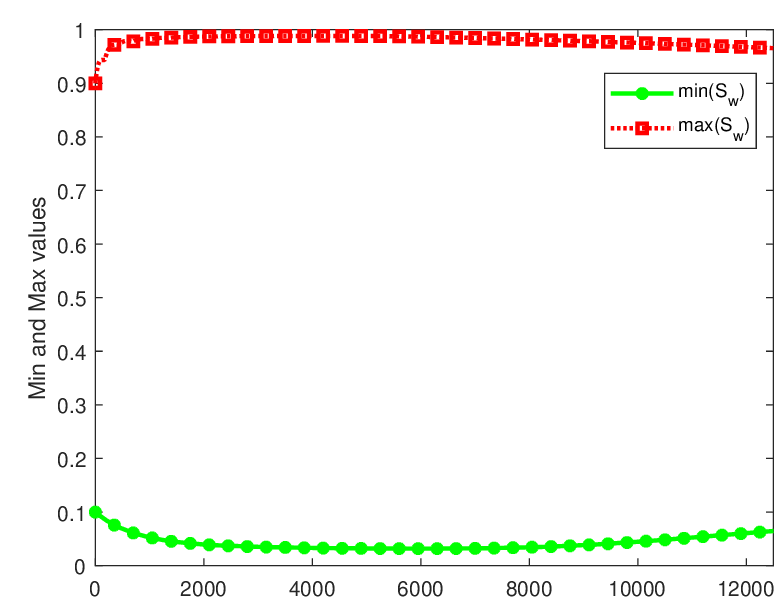}
\includegraphics[scale=0.30]{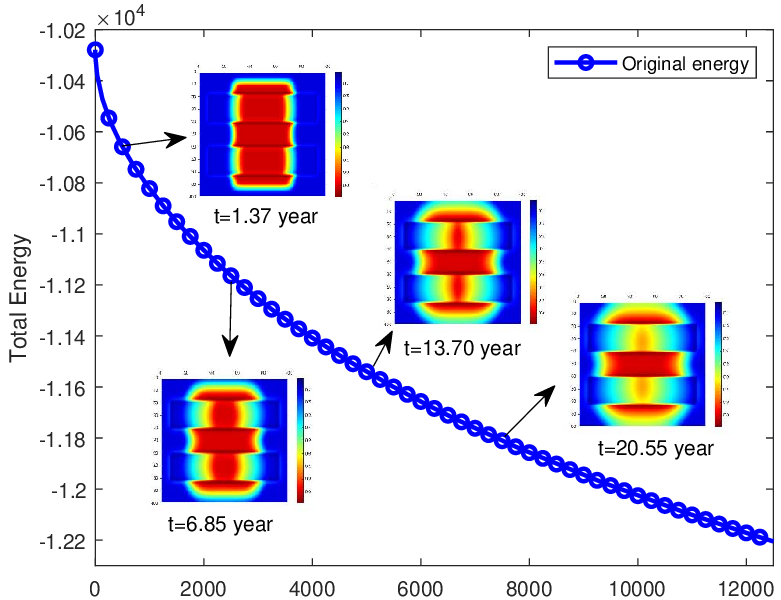}
\includegraphics[scale=0.30]{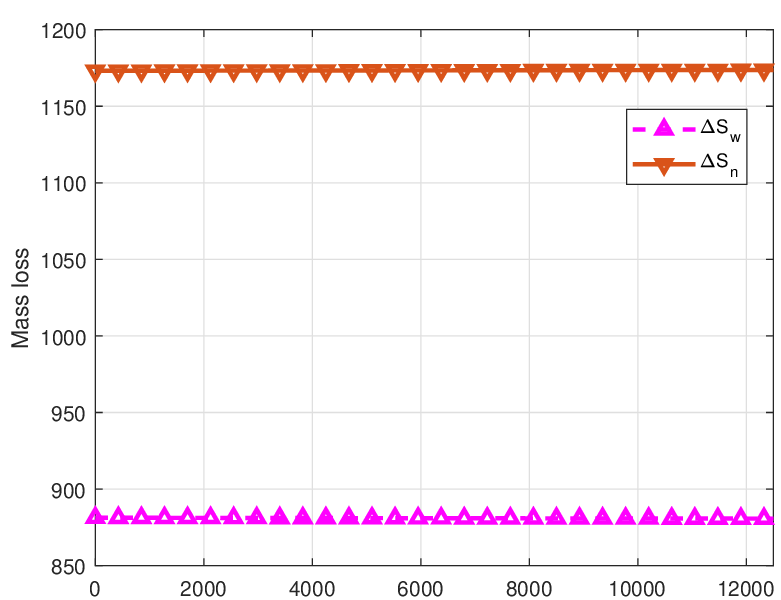}
\caption{(left) Bound of two phase saturations up to the final time step. {\color{black}(center) Original energy evolution curves up to the final time step. }(right) Mass conservation up to the final time step {\color{black} in a regular domain}. }
\label{exp2_bound}
\end{figure}


{\color{black} Furthermore, the proposed numerical schemes are applicable to geometries beyond regular domains.  We present the initial distributions of porosity, permeability and wetting-phase saturation in an L-shaped domain, and the distributions of wetting-phase saturation at different time instants are displayed in Figures \ref{exp2_ini_shape} and \ref{exp2_lshape_sw} respectively. It is clearly observed that the wetting-phase fluid is flowing into the regions with lower porosity and permeability, similar to the case with a regular rectangular domain.}

\begin{figure}[!t]
\centering
\includegraphics[scale=0.28]{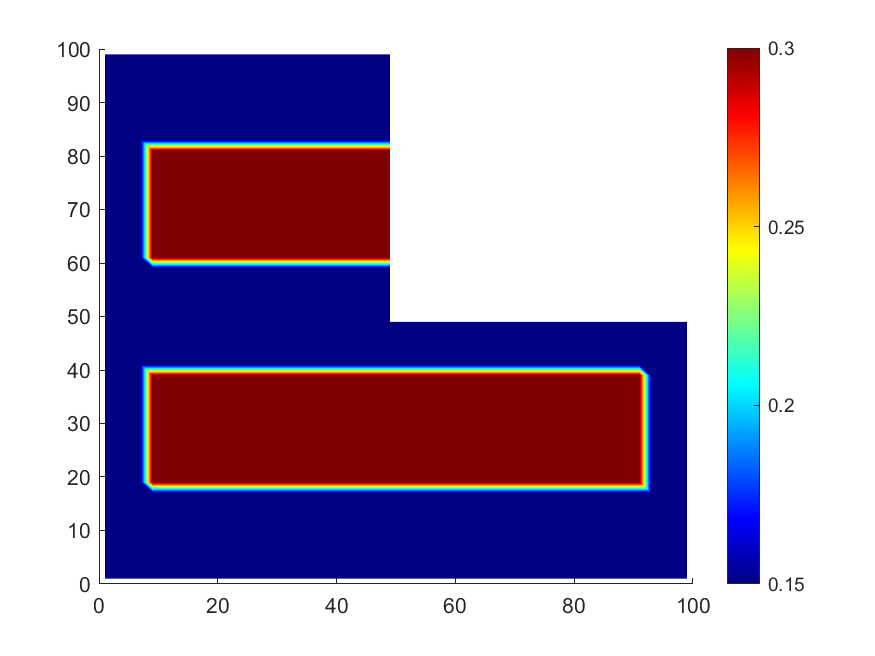} 
\includegraphics[scale=0.28]{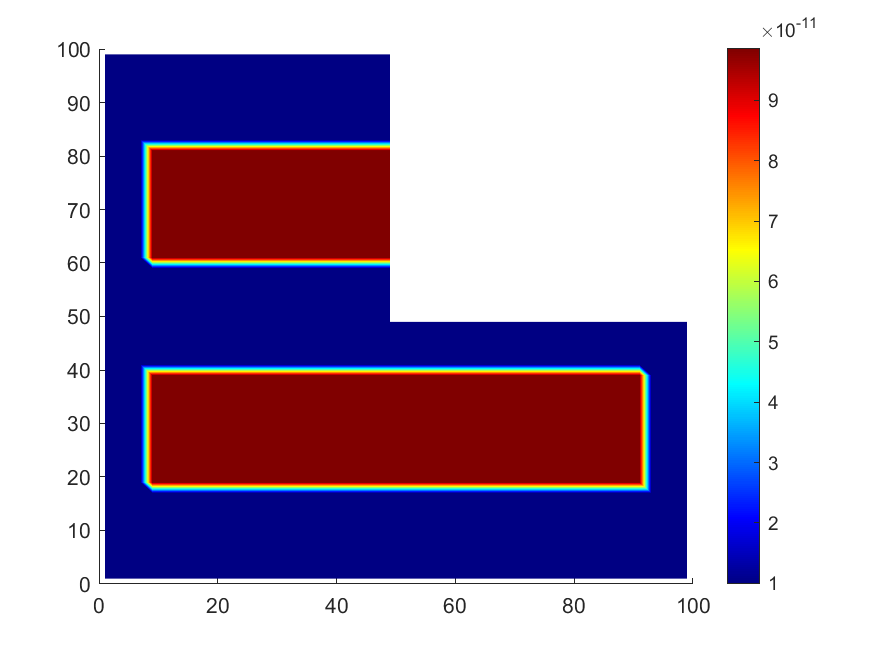}
\includegraphics[scale=0.28]{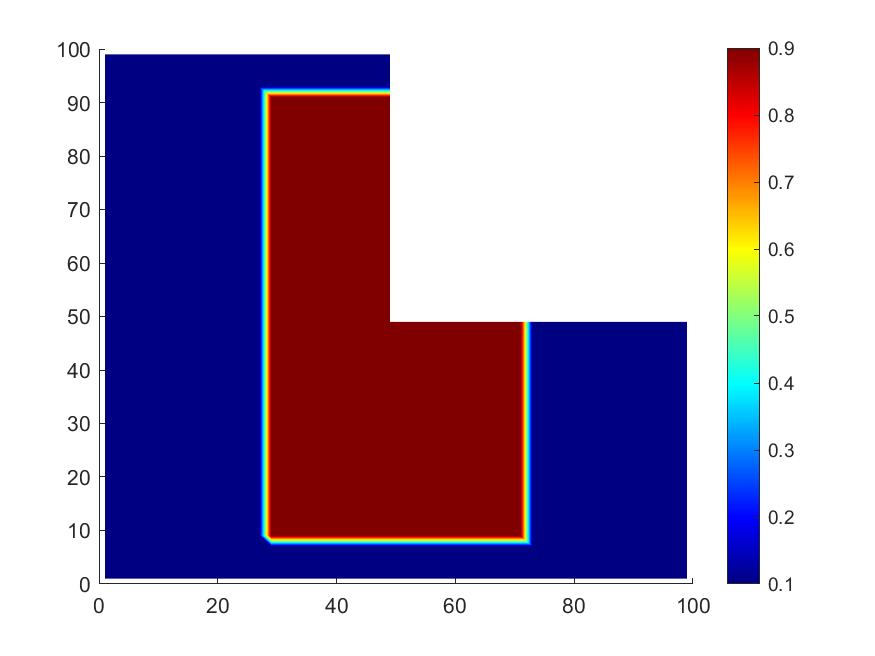}
\caption{The initial distributions of porosity(left), permeability(center) and wetting-phase saturation(right) {\color{black} in L-shaped domain}. }
\label{exp2_ini_shape}
\end{figure}

\begin{figure}[!t]
\centering
\includegraphics[scale=0.28]{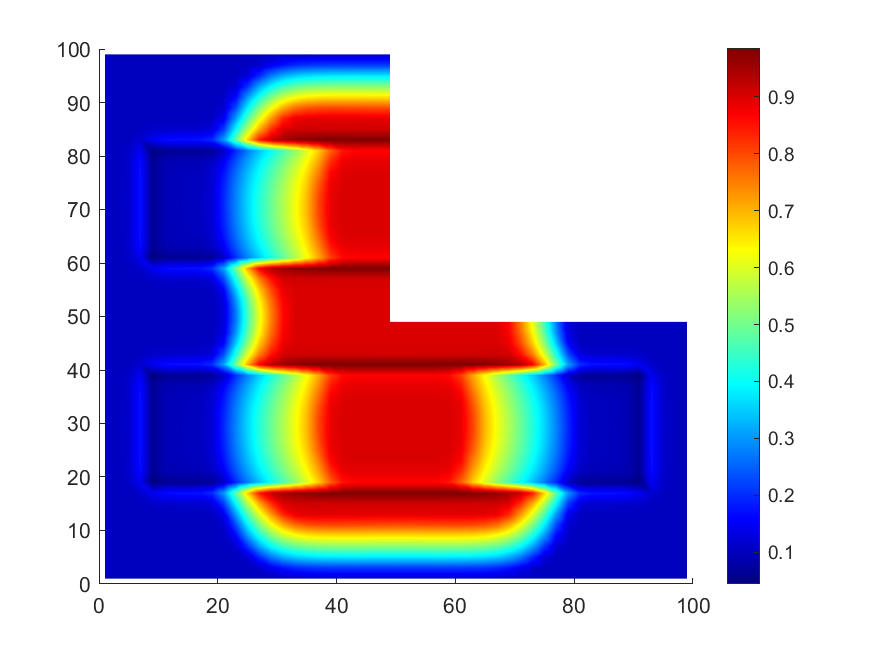} 
\includegraphics[scale=0.28]{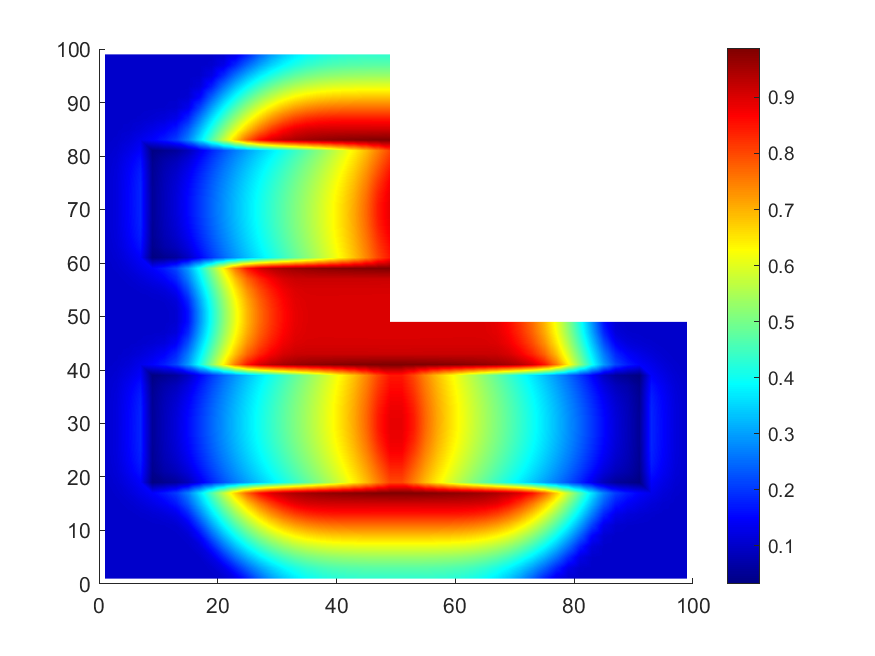}
\includegraphics[scale=0.28]{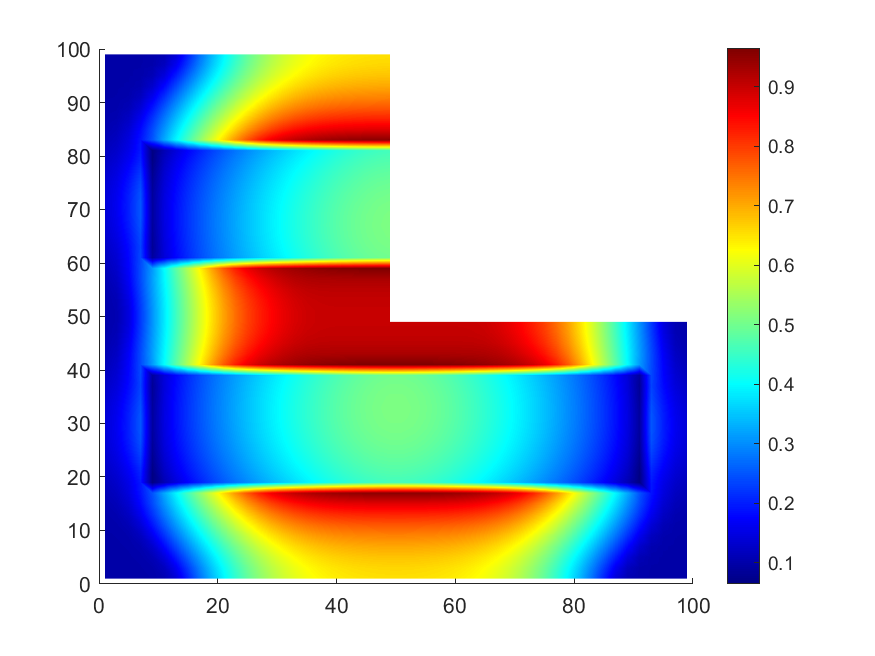}
\caption{The distributions of wetting-phase saturation at 4.11 year, 13.70 year, 34.25 year in 50 $\times$ 50 mesh {\color{black} in L-shaped domain}.}
\label{exp2_lshape_sw}
\end{figure}


\subsection{Example 3}\label{example3} 
In this example, we simulate the three-dimensional displacement process where one fluid progressively replaces another, by choosing the simulation domain as $\Omega=[0,L]^3$ with 8000 mesh cells, where $L=100m$. At the initial time, we take the wetting-phase saturation as $S_w^0=0.01$ and $\phi=0.2$. 
{\color{black} There exists a subdomain $\Omega_h$ as follows:
$$\Omega_h=\left \{ \textbf{x}=(x,y,z): 0 \le x \le  L,~ 0.2L \le y \le  0.4L  ~\text{and}~  0.6L \le y \le  0.8L , ~ 0 \le z \le  L\right \}, $$
which has a higher permeability $K = 100d$, while the permeability in the remaining domain is $K = 1d$.} We take the viscosity parameters as $\eta_{w}=1 \ cp,\eta_{n}=0.75 \ cp$. The energy parameters are given by $\sigma_{w}=2.331 \ pa$, $\sigma_{n}=0.2159 \ pa$, $\sigma_{wn}=1.4848 \ pa$ for the higher permeable layers, and $\sigma_{w}=11. 655\ pa$, $\sigma_{n}=1.0796\ pa$, $\sigma_{wn}=7.424\ pa$ for the lower permeable layers. The parameter $m$ is set to be $m=3$. The wetting phase fluid is injected from the left boundary of the medium to replace the non-wetting phase fluid flowing out at right boundary of the medium. The injection velocity is given by $\boldsymbol{u}_w=0.7\ m/year$.  There is no mass flux on the other boundaries. The pressure $p=1\ bar$ and $\nabla \mu_{\alpha} \cdot \textbf{n}=0$ are imposed on the right boundary, where \textbf{n} is the outward unit normal vector to the right boundary. The time step size is taken to be $\tau=0.8\ day$.

The distributions of wetting-phase saturation and wetting-phase rate at different time instants are depicted in Figures \ref{exp3_sw} and \ref{exp3_u} respectively. It is clearly observed that more wetting-phase fluid flows into high permeability areas, as the chemical potential gradient is a primary driving force for this process. 
\begin{figure}[!t]
\centering
\includegraphics[scale=0.38]{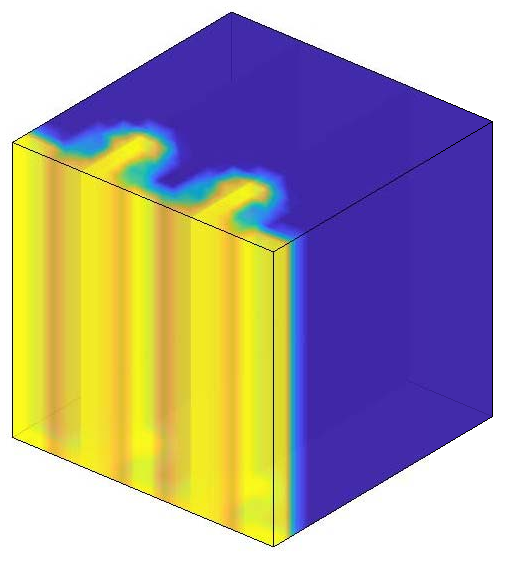}\quad
\includegraphics[scale=0.38]{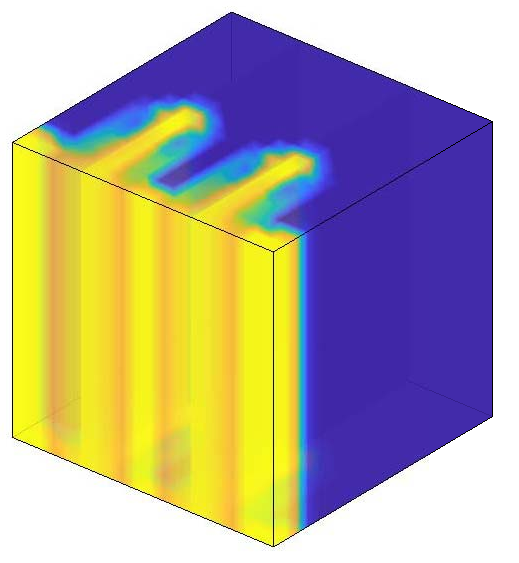}\quad
\includegraphics[scale=0.38]{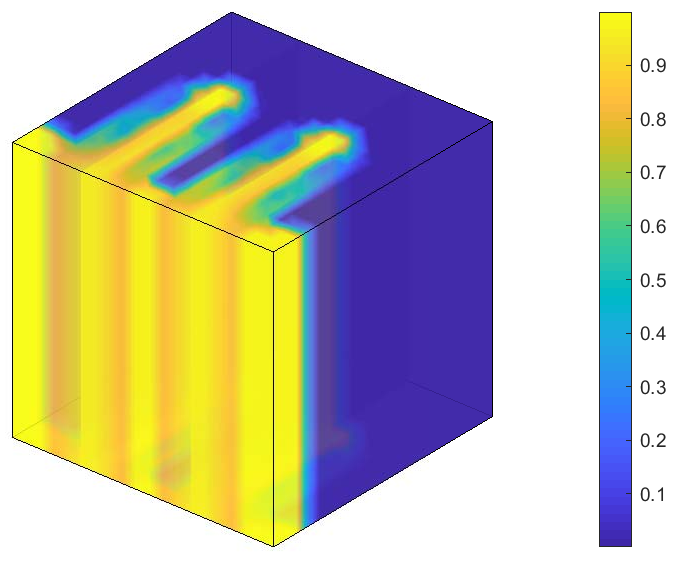}
\caption{The distributions of wetting-phase saturation at 4.59 year, 6.96 year, 8.97 year.}
\label{exp3_sw}
\end{figure}


\begin{figure}[t]
\centering
\includegraphics[scale=0.38]{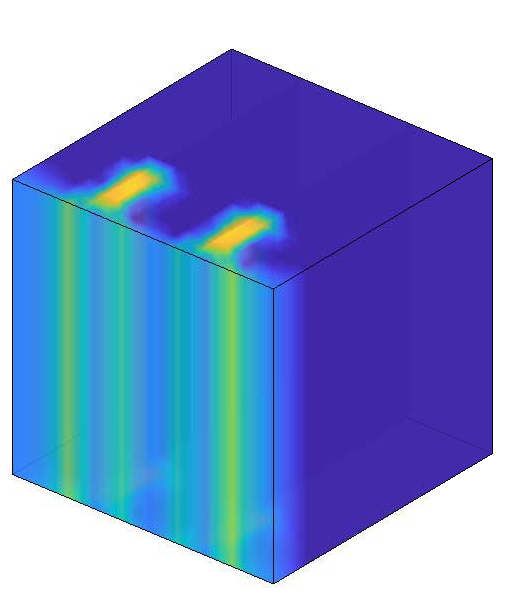}\quad
\includegraphics[scale=0.38]{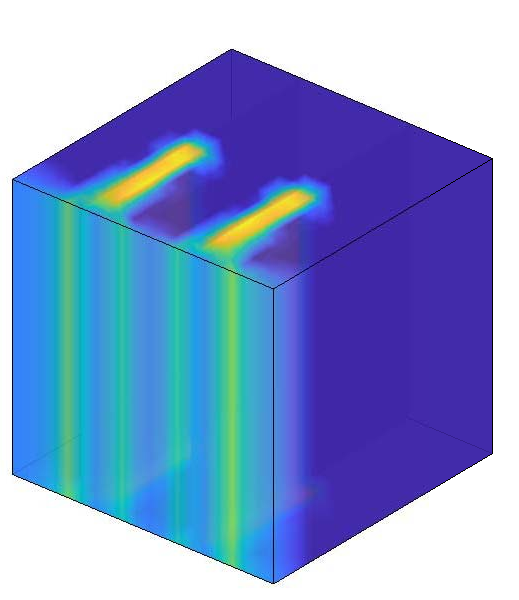}\quad
\includegraphics[scale=0.38]{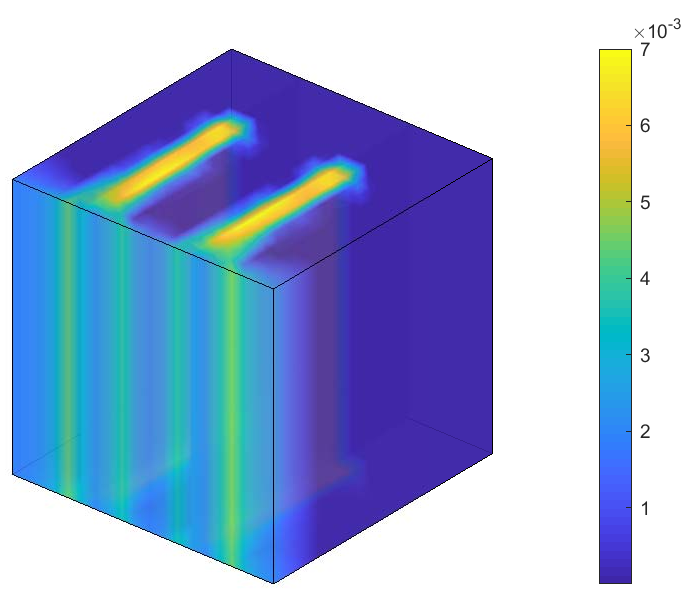}
\caption{The distributions of wetting-phase rate at 4.59 year, 6.96 year, 8.97 year.}
\label{exp3_u}
\end{figure}
\section{Conclusions}\label{section7}
In this work, we propose two efficient first- and second-order in time, fully decoupled, structure-preserving numerical schemes for two phase flow in the porous media. The unique solvability, bound-preserving property and original energy dissipation are rigorously proved. Furthermore, we establish an error estimate for the first-order scheme, which gives an optimal convergence rate for saturations of two phases.  A promising direction for future research involves an extension of these methods to more complex and compressible models, where developing second-order, bound-preserving, and energy stable numerical schemes remains a substantial challenge.

	\bibliographystyle{siamplain}
	\bibliography{two_phase}
\end{document}